\documentclass[12pt]{amsart}
\usepackage{amsmath,latexsym,amsfonts,amssymb,amsthm}
\usepackage{geometry}
\usepackage{hyperref}
\usepackage{mathrsfs}
\usepackage{graphicx,color}
\usepackage[alphabetic]{amsrefs}
\usepackage[all,cmtip]{xy}
\usepackage{bbm}

\newtheorem{prop}{Proposition}[section]
\newtheorem{thm}[prop]{Theorem}
\newtheorem{cor}[prop]{Corollary}

\newtheorem{lem}[prop]{Lemma}

\theoremstyle{definition}

\newtheorem{defn}[prop]{Definition}
\newtheorem{expl}[prop]{Example}
\newtheorem{rem}[prop]{\it Remark}
\newtheorem{emp}[prop]{}

\newtheorem*{claim*}{Claim}

\newcommand{\bR}{\mathbb{R}}
\newcommand{\bA}{\mathbb{A}}
\newcommand{\bQ}{\mathbb{Q}}
\newcommand{\bZ}{\mathbb{Z}}
\newcommand{\bN}{\mathbb{N}}

\newcommand{\bG}{\mathbb{G}}
\newcommand{\bT}{\mathbb{T}}
\newcommand{\bk}{\mathbbm{k}}

\newcommand{\oX}{\overline{X}}
\newcommand{\oY}{\overline{Y}}
\newcommand{\oM}{\overline{M}}
\newcommand{\oD}{\overline{D}}
\newcommand{\oDelta}{\overline{\Delta}}

\newcommand{\bu}{\bar{u}}
\newcommand{\bphi}{\bar{\varphi}}

\newcommand{\cX}{\mathcal{X}}
\newcommand{\cY}{\mathcal{Y}}
\newcommand{\cC}{\mathcal{C}}
\newcommand{\cO}{\mathcal{O}}

\newcommand{\cI}{\mathcal{I}}

\newcommand{\cF}{\mathcal{F}}

\newcommand{\cD}{\mathcal{D}}
\newcommand{\cP}{\mathcal{P}}
\newcommand{\cR}{\mathcal{R}}
\newcommand{\cS}{\mathcal{S}}

\newcommand{\fa}{\mathfrak{a}}

\newcommand{\fm}{\mathfrak{m}}
\newcommand{\ft}{\mathfrak{t}}

\newcommand{\Spec}{\mathrm{Spec}}

\newcommand{\Supp}{\mathrm{Supp}}
\newcommand{\Hom}{\mathrm{Hom}}
\newcommand{\mult}{\mathrm{mult}}
\newcommand{\lct}{\mathrm{lct}}

\newcommand{\vol}{\mathrm{vol}}
\newcommand{\ord}{\mathrm{ord}}
\newcommand{\Val}{\mathrm{Val}}

\newcommand{\hvol}{\widehat{\rm vol}}

\newcommand{\rup}[1]{\lceil #1\rceil}
\newcommand{\wt}{\mathrm{wt}}
\newcommand{\Coef}{\mathrm{Coeff}}

\newcommand{\mld}{\mathrm{mld}}
\newcommand{\mldk}{\mathrm{mld}^{\mathrm{K}}}

\newcommand{\gr}{\mathrm{gr}}

\newcommand{\txi}{\langle \xi \rangle}

\newcommand{\grhat}[1]{\gr_\mu (\hat{#1})}

\numberwithin{equation}{section}

\title[Boundedness of singularities]{Boundedness of log Fano cone singularities and discreteness of local volumes}

\author{Chenyang Xu}
\address{Department of Mathematics, Princeton University, Princeton, NJ 08544, USA}
\email     {chenyang@princeton.edu}

\author{Ziquan Zhuang}
\address{Department of Mathematics, Johns Hopkins University, Baltimore, MD 21218, USA}
\email{zzhuang@jhu.edu}

\date{}

\begin{document}

\begin{abstract}
We prove that in any fixed dimension, K-semistable log Fano cone singularities whose volumes are bounded from below by a fixed positive number form a bounded set. As a consequence, we show that the set of local volumes of klt singularities of a fixed dimension has zero as the only accumulation point. 
\end{abstract}

\maketitle

\section{Introduction}
In recent years, there has been remarkable progress in the algebro-geometric study of K-stability. Besides the global theory for Fano varieties, a local stability theory has also been introduced for Kawamata log terminal (klt) singularities, see \cites{LLX-nv-survey,Z-SDC-survey} for an overview. Notably, the stable degeneration conjecture is settled in \cite{XZ-SDC} (see also \cites{Blu-minimizer-exist,LX-stability-higher-rank,Xu-quasi-monomial,XZ-minimizer-unique}).  It provides for any klt singularity $x\in (X={\rm Spec}(R),\Delta)$, a canonical degeneration to a K-semistable log Fano cone singularity $x_0\in (X_0,\Delta_0;\xi_v)$. More precisely, let $v$ be a valuation minimizing the normalized volume function 
 \[
 \hvol_{X,\Delta}\colon \Val_{X,x}\to \mathbb{R}_{>0}\cup \{+\infty\}
 \]
 defined as in \cite{Li-normalized-volume}, then the corresponding degeneration is obtained as $X_0={\rm Spec}({\rm Gr}_vR)$ with a cone vertex $x_0\in X_0$, $\Delta_0$ is the corresponding degeneration of $\Delta$ and $\xi_v$ the Reeb vector induced by the valuation $v$. This degeneration is volume-preserving, i.e. it satisfies
\begin{equation}\label{eq-equivolume}
\hvol(x,X,\Delta)=\hvol_{X,\Delta}(v)=\hvol_{X_0,\Delta_0}(\wt_{\xi_v})=\hvol(x_0,X_0,\Delta_0)\,.
\end{equation}

To complete the picture of the local stability theory, one central open problem (see e.g. \cite[Conjecture 1.7]{XZ-SDC}) is the boundedness of K-semistable $n$-dimensional log Fano cone singularities $(X,\Delta;\xi)$, assuming we fix a lower bound of the local volume 
\[
\hvol(x,X,\Delta)=\hvol_{X,\Delta}(\wt_{\xi})\,.
\]
The aim of this paper is to settle this local boundedness question. Our main theorem is the following. The relevant definitions are recalled in Section \ref{s:prelim}. 
 
\begin{thm}\label{thm-Fanocone bounded}
Let $n\in \bN$ and let $I\subseteq [0,1]$ be a finite set. Let $\varepsilon>0$ be a positive real number. Let $\mathcal{K}$ be the set of K-semistable $n$-dimensional polarized log Fano cone singularities $x\in (X,\Delta;\xi)$ with coefficients in $I$ such that
\[
\hvol(x,X,\Delta)\ge \varepsilon.
\]
Then $\mathcal{K}$ is bounded.
\end{thm}

By the volume-preserving property of the degeneration \eqref{eq-equivolume}, Theorem \ref{thm-Fanocone bounded} implies that all $n$-dimensional klt singularities with a positive lower bound on their volumes, are bounded up to degeneration. In particular, we have the following consequence which gives a positive answer to \cite[Question 6.12]{LLX-nv-survey}.

\begin{thm}\label{thm-discrete volume}
Fix $n\in\bN$ and a finite set $I\subseteq[0,1]$. Then the set
\[
\widehat{\rm Vol}_{n,I}=\left\{\hvol(x,X,\Delta)\mid x\in (X,\Delta)\mbox{ is klt, }\dim(X)=n, \, {\rm Coeff}(\Delta)\subseteq I\right\}
\]
has 0 as the only accumulation point. 
\end{thm}

As another application of Theorem \ref{thm-Fanocone bounded}, we show that many classical invariants of klt singularities are controlled by their local volume.

\begin{thm} \label{thm:volume control other invariants}
Fix $n\in\bN$, $\varepsilon>0$ and a finite set $I\subseteq[0,1]$. Then there exists some  constant $M>0$ depending only on $n,\varepsilon,I$ such that for any $n$-dimensional klt singularity $x\in (X,\Delta)$ with ${\rm Coeff}(\Delta)\subseteq I$ and $\hvol(x,X,\Delta)\ge \varepsilon$, we have
\[
\dim_{\kappa(x)} \fm_x/\fm_x^2 \le M, \quad \mult_x X\le M, \quad\mldk(x,X,\Delta)\le M \, .
\]
\end{thm}

Here $\mld^K(x,X,\Delta)$ denotes the minimal log discrepancy of the Koll\'ar components of the singularity, see Definition \ref{defn:mld^K}. For the embedded dimension $\dim_{\kappa(x)} \fm_x/\fm_x^2$ and the multiplicity $\mult_x X$, the existence of the uniform upper bound $M$ directly follows from \eqref{eq-equivolume}, Theorem \ref{thm-Fanocone bounded} and the upper semi-continuity of these invariants. The $\mld^K$ case is more subtle, see Theorem \ref{thm:mld^K bdd}. Note that the converse of Theorem \ref{thm:volume control other invariants} is false, i.e. even if the embedded dimension, multiplicity and $\mld^K$ of the singularities are all bounded from above, the local volume of the singularity can still be arbitrarily small. This can be seen already in the case of ADE surface singularities.

Philosophically, one can compare Theorem \ref{thm-discrete volume} with \cite[Theorem 1.3]{HMX-ACC} which deals with the global case of log general type pairs, where volumes might accumulate even when the coefficients belong to a finite set, though they still satisfy the descend chain condition (DCC). Special cases of Theorem \ref{thm-discrete volume} have been previously established, including when $X$ is bounded \cite{HLQ-vol-ACC},  $(X,\Delta)$ is of complexity at most one \cites{MS-bdd-toric,LMS-bdd-dim-3}, or when $X$ is of dimension at most three \cites{LMS-bdd-dim-3,Z-mld^K-2}. 

To prove Theorem \ref{thm-Fanocone bounded}, we need to sort out a birational geometric condition that is more flexible than K-semistability so that it is preserved under small perturbations of the polarization. In particular, we want to include the case when $\xi$ is a rational perturbation of the Reeb vector coming from the K-semistable log Fano cone structure. This is why instead of proving Theorem \ref{thm-Fanocone bounded} directly, we aim to prove the following more general statement.

\begin{thm} \label{thm:polarized cone bdd}
Let $n\in \bN$ and let $I\subseteq [0,1]$ be a finite set. Let $\varepsilon,\theta>0$. Let $\cS$ be the set of $n$-dimensional polarized log Fano cone singularities $x\in (X,\Delta;\xi)$ with coefficients in $I$ such that
\[
\hvol(x,X,\Delta)\ge \varepsilon,\quad\mathrm{and}\quad 
\Theta(X,\Delta;\xi)\ge \theta.
\]
Then $\cS$ is bounded.
\end{thm}

Here $\Theta(X,\Delta;\xi)$ is the volume ratio of the log Fano cone singularity, see Definition \ref{def:vol ratio}. It serves as a local analog of the stability invariant of Fano varieties. A polarized log Fano cone singularity $x\in (X,\Delta;\xi)$ is K-semistable if and only if $\Theta(X,\Delta;\xi)=1$, thus Theorem \ref{thm:polarized cone bdd} implies Theorem \ref{thm-Fanocone bounded}. An observation from \cite{Z-mld^K-2} is that the condition of volume ratio having a lower bound should be the right generalization of the K-semistable condition to guarantee boundedness. 

Under the assumption that $\Theta(X,\Delta;\xi)\ge \theta$, the condition that $\hvol(x,X,\Delta)$ has a uniform positive lower bound is equivalent to $\hvol_{X,\Delta}(\wt_{\xi})$ being uniformly bounded from below by a positive number. When $\xi$ is rational, the latter condition is equivalent to a global condition on the orbifold base.  This leads to a special case for this version of Theorem \ref{thm:polarized cone bdd}.

\begin{cor}\label{c-bounded cone}
Fix a positive integer $n$, and two positive numbers $\alpha_0,\varepsilon$. Consider all $(n-1)$-dimensional log Fano pairs $(V,\Delta_V)$ such that there exist some $r>0$ and some Weil divisor $L$ on $V$ which satisfy 
\[
\alpha(V,\Delta_V)\ge \alpha_0, \, -(K_V+\Delta_V)\sim_\bQ rL \mbox{ \ \ and \ \ } r(-K_V-\Delta_V)^{n-1}\ge \varepsilon\,.
\]
Let $\cS$ be the set of log Fano cone singularities given by
\[
X={\rm Spec}\bigoplus_m H^0(V,mL)
\]
for all possible $(V,\Delta,L)$ as above and $\Delta$ the closure of the pull back of $\Delta_V$ on $X$.
Then $\cS$ is bounded. 
\end{cor}

We note that, somewhat surprisingly, in Corollary \ref{c-bounded cone} all such $(V,\Delta,L)$ themselves are not bounded.

\begin{emp}[Strategy of the proof] \label{par:strategy}
The corresponding global result of
Theorem \ref{thm:polarized cone bdd} has been proved in \cite{Jia-Kss-Fano-bdd}, where it is shown that Fano varieties whose volume and alpha invariant are bounded away from zero form a bounded set. To prove the boundedness of the log Fano cone singularities $x\in (X,\Delta;\xi)$, one wants to reduce it to some boundedness question for projective Fano type varieties. 
A natural candidate would be the quotient
\[
 (V,\Delta_V):=((X,\Delta)\setminus \{x\})/\langle \xi\rangle
\]
(we may assume $\xi$ is rational after a perturbation and hence it generates a $\bG_m$-action). While one can show that the alpha invariant of the log Fano pair $(V,\Delta_V)$ is bounded from below (see \cites{Z-mld^K-1,Z-mld^K-2}), however, the volume $(-K_V-\Delta_V)^{n-1}$ could be arbitrarily small (see Example \ref{exam-weightedprojective}). In particular, $(V,\Delta_V)$ is not necessarily bounded, which posts a major challenge in the proof. 

A better candidate, first proposed in \cite{Z-mld^K-2}, is the projective orbifold cone compactification $(\overline{X},\overline{\Delta})$ of $(X,\Delta)$, which adds a divisor isomorphic to $V$ at infinity. One piece of evidence from \cite{Z-mld^K-2} is that if the local volume of the singularity $x\in (X,\Delta)$ is bounded from below, then so is the volume of log Fano pair $(\overline{X},\overline{\Delta})$. On the other hand, the alpha invariant of $(\overline{X},\overline{\Delta})$ can be arbitrarily small and the log Fano pairs $(\overline{X},\overline{\Delta})$ still do not form a bounded family as we vary the Reeb vector $\xi$. To get around this problem, the arguments in \cite{Z-mld^K-2} rely on an additional subtle assumption on the Koll\'ar components of the singularities, and it is not clear if this extra assumption is always satisfied. In this paper, we take a different path and follow a strategy which in spirit is closer to \cites{Birkar-bab-1,Jia-Kss-Fano-bdd}.

The first step is to prove a birational version of \cite{Jia-Kss-Fano-bdd}. More specifically, we first show (see Section \ref{ss:orbifold}) that if the volume ratio $\Theta(X,\Delta;\xi)$ is bounded from below, then away from the divisor at infinity, the alpha invariant of $(\oX,\oDelta)$ is also bounded from below, i.e. there is some uniform $\alpha_0>0$ such that $\alpha_{x_1}(\oX,\oDelta)\ge \alpha_0$ for all $x_1\in X$. By adapting the arguments of \cite{Jia-Kss-Fano-bdd}, we then show (see Section \ref{ss:log bir bdd}) that this together with boundedness of the volume imply that the set of projective orbifold cones $(\oX,\oDelta)$ is log birationally bounded (Definition \ref{defn:birational dominate sets}). In particular, they are birational to a bounded set of pairs.

The next step is to improve the log birational boundedness to boundedness in codimension one, see Section \ref{ss:bdd in codim 1}. Inspired by \cite{Birkar-bab-1}, and using the fact that $\alpha_{x_1}(\oX,\oDelta)\ge \alpha_0$ away from the divisor at infinity, we show that there is a uniform way to modify the bounded birational model $Y$ obtained from the previous step, so that the only exceptional divisor of the induced birational map  $\oX\dashrightarrow Y$ is the divisor at infinity. In other words, $Y\dashrightarrow \oX$ is close to a birational contraction except possibly over one divisor. This is the best we can hope for, as the divisor at infinity depends on the Reeb vector, and in general cannot be extracted on a bounded model. The main ingredient for this step is the construction of sub-klt bounded complements of $(\oX,\oDelta)$ with certain control on the negative part. This in turn relies on the boundedness of complements proved in \cite[Theorem 1.7]{Birkar-bab-1}, as well as the birationally bounded model constructed in the previous step.

Finally, to finish the argument,  we recover $(X,\Delta)$ by running a carefully chosen minimal model program on $Y$ and using the affineness of $X$ to show that $X$ is embedded as an open set of the ample model obtained from the minimal model program sequence. See Section \ref{ss:MMP step}.
\end{emp}

\noindent {\bf Acknowledgement.} CX is partially supported by NSF DMS-2139613, DMS-2201349 and a Simons Investigator grant. ZZ is partially supported by the NSF Grants DMS-2240926, DMS-2234736, a Clay research fellowship, as well as a Sloan fellowship.

\section{Preliminaries} \label{s:prelim}

\subsection{Notation and conventions}

We work over an algebraically closed field $\bk$ of characteristic $0$. We follow the standard terminology from \cites{KM98,Kol13}.

A sub-pair $(X,\Delta)$ consists of a normal variety $X$ together with an $\bR$-divisor $\Delta$ on $X$ (a priori, we do not require that $K_X+\Delta$ is $\bR$-Cartier). It is called a pair if $\Delta$ is effective. A log smooth pair $(Y,\Sigma)$ consists of a smooth variety $Y$ and a simple normal crossing divisor $\Sigma$ on $Y$. We say that a pair $(X,\Delta)$ is log Fano if $(X,\Delta)$ is klt (\cite[Definition 2.34]{KM98}) and $-(K_X+\Delta)$ is $\bR$-Cartier and ample. 

A singularity $x\in (X,\Delta)$ consists of a pair $(X,\Delta)$ and a closed point $x\in X$. We will always assume that $X$ is affine and $x\in \Supp(\Delta)$ (whenever $\Delta\neq 0$). We say that the singularity is klt if $(X,\Delta)$ is klt in a neighbourhood of $x$.

Given an $\bR$-divisor $\Delta$ on $X$ and a birational map $\varphi\colon Y\dashrightarrow X$, we denote the strict transform of $\Delta$ on the birational model $Y$ by $\Delta_Y$, i.e. $\Delta_Y=\varphi_*^{-1}\Delta$. If $\Delta$ is $\bR$-Cartier, the birational pullback $\varphi^*\Delta$ is defined as the $\bR$-divisor $f_*g^*\Delta$ where $f\colon W\to Y$, $g\colon W\to X$ is a common resolution.

When we refer to a constant $C$ as $C=C(n,\varepsilon,\cdots)$ it means $C$ only depends on $n,\varepsilon,\cdots$, etc.

\subsection{Local volumes} \label{ss:local vol}

We first briefly recall the definition of the local volumes of klt singularities \cite{Li-normalized-volume}. For this we need the notion of valuations. A \emph{valuation} over a singularity $x\in X$ is an $\bR$-valued valuation $v: K(X)^* \to \bR$ (where $K(X)$ denotes the function field of $X$) such that $v$ is centered at $x$ (i.e. if $f\in \cO_{X,x}$, then $v(f)>0$ if and only if $f\in \fm_x$) and $v|_{\bk^*}=0$. The set of such valuations is denoted as $\Val_{X,x}$. Let $x\in (X,\Delta)$ be a singularity and assume that $K_X+\Delta$ is $\bR$-Cartier. The \emph{log discrepancy} function
\[
A_{X,\Delta}\colon \Val_{X,x}\to \bR \cup\{+\infty\},
\]
is defined as in \cite{JM-val-ideal-seq} and \cite[Theorem 3.1]{BdFFU-log-discrepancy}. It generalizes the usual log discrepancies of divisors; in particular, for divisorial valuations, i.e. valuations of the form $\lambda\cdot \ord_F$ where $\lambda>0$ and $F$ is a prime divisor on some proper birational model $\pi\colon Y\to X$, we have 
\[
A_{X,\Delta}(\lambda\cdot \ord_F)=\lambda\cdot A_{X,\Delta}(F)=\lambda\cdot (1+\ord_F(K_Y-\pi^*(K_X+\Delta))).
\]
For klt singularities, one has $A_{X,\Delta}(v)>0$ for all $v\in \Val_{X,x}$. We denote by $\Val^*_{X,x}$ the set of valuations $v\in\Val_X$ with center $x$ and  $A_{X,\Delta}(v)<+\infty$. The \emph{volume} of a valuation $v\in\Val_{X,x}$ is defined as
\[
\vol(v)=\vol_{X,x}(v)=\limsup_{m\to\infty}\frac{\ell(\cO_{X,x}/\fa_m(v))}{m^n/n!},
\] 
where $n=\dim X$ and $\fa_m (v)$ denotes the valuation ideal, i.e.
\[
\fa_m (v):=\{f\in \cO_{X,x}\mid v(f)\ge m\}. 
\]

\begin{defn} \label{defn:local volume}
Let $x\in (X,\Delta)$ be an $n$-dimensional klt singularity. For any $v\in \Val_{X,x}$, we define the \emph{normalized volume} of $v$ as
\[
\hvol_{X,\Delta}(v):=
\begin{cases}
A_{X,\Delta}(v)^n\cdot\vol_{X,x}(v) \text{ \ \ \ \  if $A_{X,\Delta}(v)<+\infty$} \\
+\infty \text{ \ \ \ \ \ \ \ \ \ \ \ \ \ \ \ \ \ \ \ \  \ \  \ \  if $A_{X,\Delta}(v)=+\infty$}
                    \end{cases}\,.
\]
The \emph{local volume} of $x\in (X,\Delta)$ is defined as
\[
  \hvol(x,X,\Delta):=\inf_{v\in\Val^*_{X,x}} \hvol_{X,\Delta}(v)\, .
\]
\end{defn}

By \cite[Theorem 1.2]{Li-normalized-volume}, the local volume of a klt singularity is always positive.

\subsection{Log Fano cone singularities} \label{ss:log Fano cone}

Next we recall the definition of log Fano cone singularities and their K-semistability. 

\begin{defn}
Let $X=\Spec(R)$ be a normal affine variety and $\bT$ an algebraic torus (i.e. $\bT\cong \bG_m^r$ for some $r>0$). We say that a $\bT$-action on $X$ is {\it good} if it is effective and there is a unique closed point $x\in X$ that is in the orbit closure of any $\bT$-orbit. We call $x$ the vertex of the $\bT$-variety $X$. 
\end{defn}

Let $N:=N(\bT)=\Hom(\bG_m, \bT)$ be the co-weight lattice and $M=N^*$ the weight lattice. We have a weight decomposition 
\[
R=\oplus_{\alpha\in M} R_\alpha,
\]
and the action being good implies that $R_0=\bk$ and every $R_\alpha$ is finite dimensional. For $f\in R$, we denote by $f_\alpha$ the corresponding component in the above weight decomposition.

\begin{defn}
A \emph{Reeb vector} on $X$ is a vector $\xi\in N_\bR$ such that $\langle \xi, \alpha \rangle>0$ for all $0\neq \alpha\in M$ with $R_{\alpha}\neq 0$. The set $\ft^+_{\bR}$ of Reeb vectors is called the Reeb cone.
\end{defn}

For any $\xi\in \ft^+_{\bR}$, we can define a valuation $\wt_\xi$ by setting
\[
\wt_\xi (f):=\min\{\langle \xi, \alpha \rangle\mid \alpha\in M, f_\alpha\neq 0\}
\]
where $f\in R$. It is not hard to verify that $\wt_\xi\in \Val_{X,x}$.

\begin{defn}
A log Fano cone singularity is a klt singularity that admits a nontrivial good torus action. A polarized log Fano cone singularity $x\in (X,\Delta;\xi)$ consists of a log Fano cone singularity $x\in (X,\Delta)$ together with a Reeb vector $\xi$ (called a polarization).
\end{defn}

By abuse of convention, a good $\bT$-action on a klt singularity $x\in (X,\Delta)$ means a good $\bT$-action on $X$ such that $x$ is the vertex and $\Delta$ is $\bT$-invariant. Using terminology from Sasakian geometry, we say a polarized log Fano cone $x\in (X,\Delta;\xi)$ is \emph{quasi-regular} if $\xi$ generates a $\bG_m$-action (i.e. $\xi$ is a real multiple of some element of $N$); otherwise, we say that $x\in (X,\Delta;\xi)$ is \emph{irregular}. The torus generated by $\xi$ will be denoted by $\txi$.

\begin{defn}[{\cites{CS-Kss-Sasaki,CS-Sasaki-Einstein}, \cite[Theorem 2.34]{LX-stability-higher-rank}}]
We say that a polarized log Fano cone singularity $x\in (X,\Delta;\xi)$ is \mbox{\emph{K-semistable}} if 
\[
\hvol(x,X,\Delta)=\hvol_{X,\Delta}(\wt_\xi).
\]
\end{defn}

Log Fano cone singularities play a special role in the local stability theory of klt singularities, due to the following statement. It was originally known as the Stable Degeneration Conjecture.

\begin{thm}\label{thm:SDC}
Every klt singularity $x\in (X={\rm Spec}(R),\Delta)$ has a special degeneration to a K-semistable log Fano cone singularity $x_0\in (X_0,\Delta_0;\xi_v)$ with 
\[\hvol(x,X,\Delta)=\hvol(x_0,X_0,\Delta_0).\]
More precisely, up to rescaling, there is a unique valuation $v$ minimizing $\hvol_{X,\Delta}$, and $X_0={\rm Spec}({\rm Gr}_v(R))$. In addition, denote by $\Delta_0$ the degeneration of $\Delta$, and $\xi_v$  the Reeb vector induced by $v$, then $(X_0,\Delta_0;\xi_v)$ is a K-semistable log Fano cone. 
\end{thm}

We call $x_0\in (X_0,\Delta_0;\xi_v)$ the K-semistable log Fano cone degeneration of $x\in (X,\Delta)$. It is unique up to the rescaling of the Reeb vector $\xi_v$.

\begin{proof}
See \cites{Blu-minimizer-exist,LX-stability-higher-rank,Xu-quasi-monomial,XZ-minimizer-unique,XZ-SDC} for the case of rational coefficients and the extension to real coefficients in \cite{Z-mld^K-2}.
\end{proof}

\begin{defn} \label{def:vol ratio}
The \emph{volume ratio} of a polarized log Fano cone singularity $x\in (X,\Delta;
\xi)$ is defined to be
\[
\Theta(X,\Delta;
\xi):=\frac{\hvol(x,X,\Delta)}{\hvol_{X,\Delta}(\wt_\xi)}.
\]
\end{defn}

By definition, $0<\Theta(X,\Delta;
\xi)\le 1$ and $\Theta(X,\Delta;
\xi)=1$ if and only if $x\in (X,\Delta;
\xi)$ is K-semistable.

\subsection{Orbifold cones} \label{ss:orbifold}

Every quasi-regular polarized log Fano cone singularity has a natural projective orbifold cone compactification. This provides a convenient way to think about these singularities. In this subsection, we recall this construction and relate some invariants of the singularities with those of the projective orbifold cones. For more background, see \cite{Kol-Seifert-bundle} or \cite[Section 3.1]{Z-mld^K-2}. 

Let $x\in (X,\Delta;\xi)$ be a quasi-regular log Fano cone singularity. By definition, the $\txi \cong \bG_m$-action on $X\setminus \{x\}$ has finite stabilizers, hence the quotient map 
\[X\setminus \{x\} \to V:=(X\setminus \{x\})/\txi\] 
is a Seifert $\bG_m$-bundle (in the sense of \cite{Kol-Seifert-bundle}), and there is an ample $\bQ$-divisor $L$ on $V$ such that (see \cite[Theorem 7]{Kol-Seifert-bundle})
\[
X={\rm Spec}\bigoplus_{m\in \bN}H^0(V, \lfloor mL \rfloor) \, .
\]
The zero section $V_0$ of this Seifert $\bG_m$-bundle gets contracted to the closed point $x\in X$ (as a valuation, $\ord_{V_0}$ is proportional to $\wt_\xi$), thus we can compactify $X$ to $\oX$ by adding the infinity section $V_{\infty}$. Let $\oDelta$ be the closure of $\Delta$ on $\oX$. We call $(\oX,\oDelta+V_{\infty})$ the (\emph{projective}) \emph{orbifold cone compactification} of $x\in (X,\Delta;\xi)$ ({\it cf.} \cite[Section 3.1]{Z-mld^K-2}).

Note that $(\oX,\oDelta+V_\infty)$ is plt and $-(K_{\oX}+\oDelta+V_{\infty})$ is ample, see \cite[Lemma 3.3]{Z-mld^K-2} or \cite{Kol-Seifert-bundle}. By adjunction along $V_\infty\cong V$, we may write 
\[
(K_{\oX}+\oDelta+V_{\infty})|_{V_{\infty}} = K_V+\Delta_V
\]
for some effective divisor $\Delta_V$. Then $(V,\Delta_V)$ is a klt log Fano pair. We call $(V,\Delta_V)$ the orbifold base of the singularity $x\in (X,\Delta;\xi)$. There exists some $r>0$ such that $-(K_{V}+\Delta_V)\sim_{\bR} rL$, and we have 
\begin{equation}\label{eq-cone polarization}
-(K_{\oX}+\oDelta+V_{\infty})\sim_\bR rV_{\infty} \, .
\end{equation}

A subtle feature of the local boundedness problem is that the orbifold bases do not belong to a bounded set; already their volumes can be arbitrarily small when we fix the singularity $x\in (X,\Delta)$ and vary the Reeb vector $\xi$.

\begin{expl}\label{exam-weightedprojective}
Let $(x\in X)=(0\in \bA^n)$ and $\xi=(\xi_1,\dots,\xi_n)$ for some pairwise coprime positive integers $\xi_i$ $(1\le i\le n)$. Assume that $n\ge 3$ and $\xi_1\le \cdots \le \xi_n$. Then 
\[ 
\oX=\mathbb{P}(1,\xi_1,\dots,\xi_n),\quad V\cong \mathbb{P}(\xi_1,\dots,\xi_n),\quad \Delta_V=0,\quad \mathrm{and} \quad L=\cO(1).
\]
We can easily compute 
\[
\Theta(\bA^n;\xi)=\frac{n^n}{\hvol_{\bA^n}({\rm wt}(\xi))}=\frac{\xi_1\cdots \xi_n\cdot n^n}{(\xi_1+\cdots+\xi_n)^n} \, .
\]
So $\Theta(\bA^n,\xi)$ has a positive lower bound if and only if $\frac{\xi_n}{\xi_1}$ has an upper bound. Using \cite[Corollary 7.16]{BJ-delta}, one can show that this is also equivalent to the condition that the $\alpha$-invariant $\alpha(V)$ defined below in Definition \ref{def:alpha} has a positive lower bound.

On the other hand, 
\[
\vol(-K_V)=\frac{(\xi_1+\cdots+\xi_n)^{n-1}}{\xi_1\cdots \xi_n }\, .
\]
So if $\frac{\xi_n}{\xi_1}$ is bounded from above, then $\vol(-K_V)$ is bounded away from zero if and only if all the weights $\xi_i$ are bounded from above.  
\end{expl}

A key observation from \cite{Z-mld^K-2}, following a direct calculation using \eqref{eq-cone polarization}, is that the volume of a log Fano cone singularity is more closely related to the global volume of its projective orbifold cone compactification (rather than the orbifold base).

\begin{lem} \label{lem:Sasaki vol as global vol}
Notation as above. Then we have
\[
\hvol_{X,\Delta}(\wt_\xi)=\vol(-(K_{\oX}+\oDelta+V_{\infty})).
\]
In particular, if $\Theta(X,\Delta;\xi)\ge \theta>0$, then $\vol(-(K_{\oX}+\oDelta+V_{\infty}))\le n^n \theta^{-1}$.
\end{lem}

\begin{proof}
The equality is \cite[Lemma 3.4]{Z-mld^K-2}. The other implication then follows from \cite[Theorem 1.6]{LX-cubic-3fold}.
\end{proof}

We next relate the volume ratio with the $\alpha$-invariants of the orbifold base or the projective orbifold cone. First we recall some definitions.

\begin{defn} \label{def:alpha}
Let $(X,\Delta)$ be a projective klt pair and let $L$ be a big $\bR$-Cartier $\bR$-divisor. We define the $\alpha$-invariant $\alpha(X,\Delta;L)$ as
\[
\alpha(X,\Delta;L):=\inf\left\{\lct(X,\Delta;D) \mid 0\le D\sim_{\bR} L \right\},
\]
where $\lct (X,\Delta;D)$ denotes the log canonical threshold, i.e. the largest number $\lambda$ such that $(X,\Delta+\lambda D)$ is log canonical.
For any projective pair $(X,\Delta)$ that is klt at a closed point $x\in X$, we can similarly define the log canonical threshold $\lct_x (X,\Delta;D)$ at $x$ and the local $\alpha$-invariant
\[
\alpha_x(X,\Delta;L):=\inf\left\{\lct_x (X,\Delta;D) \mid 0\le D\sim_{\bR} L \right\}.
\]
When the pair $(X,\Delta)$ is clear from the context, we will just write $\alpha(L)$ and $\alpha_x(L)$. For a log Fano pair $(X,\Delta)$, we define $\alpha(X,\Delta):=\alpha(X,\Delta;-K_X-\Delta)$ and similarly $\alpha_x (X,\Delta)$.
\end{defn}

While at a point on the infinity divisor $V_{\infty}$, the $\alpha$-invariant of the projective orbifold cone $(\oX,\oDelta)$ could be very small when $r$ is large in \eqref{eq-cone polarization}, the following result roughly says that for any point outside the infinity divisor, the local $\alpha$-invariant of $(\oX,\oDelta)$ is bounded by the (global) $\alpha$-invariant of the orbifold base.

\begin{lem} \label{lem:compare alpha at vertex and infinity}
Let $x\in (X,\Delta;\xi)$ be a quasi-regular polarized log Fano cone singularity. Let $(\oX,\oDelta+V_{\infty})$ be its projective orbifold cone compactification, and let $(V,\Delta_V)$ be the orbifold base. Then we have 
\[
\alpha_{x_1}(\oX,\oDelta+V_{\infty}) \ge \min\{1,\alpha(V,\Delta_V)\}
\]
for all closed point $x_1\in X$. In particular, $\alpha_x(\oX,\oDelta+V_{\infty}) = \min\{1,\alpha(V,\Delta_V)\}$.
\end{lem}

\begin{proof}
Let $D_V\sim_\bR -(K_V+\Delta_V)$ be an effective $\bR$-divisor and let $D$ be the closure in $\oX$ of its pullback to $X\setminus\{x\}$. Then we have $D\sim_\bR -(K_{\oX}+\oDelta+V_{\infty})$ and $\ord_{V_0} (D) = A_{X,\Delta}(V_0)$ (for usual cones see \cite[Proposition 3.14]{Kol13}, the general case follows from the computations in \cite[Section 4]{Kol-Seifert-bundle}). In particular, $\alpha_x(\oX,\oDelta+V_{\infty}) \le 1$. Since $X\setminus\{x\}\to V$ is a Seifert $\bG_m$-bundle, the pair $(V,\Delta_V+tD_V)$ is log canonical if and only if $(\oX,\oDelta+tD)$ is log canonical on $X\setminus\{x\}$. Thus we also get $\alpha_x(\oX,\oDelta+V_{\infty})\le \alpha(V,\Delta_V)$.

Suppose that there exists some effective $\bR$-divisor $D\sim_\bR -(K_{\oX}+\oDelta+V_{\infty})$ such that $t:=\lct_{x_1} (\oX,\oDelta+V_{\infty};D)<\min\{1,\alpha(V,\Delta_V)\}$ for some $x_1\in X$. Since $x_1\not\in V_\infty$ and $V_\infty$ is ample, we may assume that $V_\infty\not\in \Supp(D)$. The non-klt locus of the pair $(\oX,\oDelta+V_{\infty}+tD)$ thus contains at least $x_1$ and $V_\infty$. Since $-(K_{\oX}+\oDelta+V_{\infty}+tD)$ is ample, Koll\'ar-Shokurov's connectedness lemma implies that $(\oX,\oDelta+V_{\infty}+tD)$ is not plt along $V_\infty$. It then follows from adjunction that $(V,\Delta_V+tD|_{V_\infty})$ (we identify $V$ with $V_\infty$) is not klt, and hence $\alpha(V,\Delta_V)\le t$, a contradiction. In other words, we have 
\[\alpha_{x_1}(\oX,\oDelta+V_{\infty}) \ge \min\{1,\alpha(V,\Delta_V)\}.
\]
Combined with the upper bounds of $\alpha_x(\oX,\oDelta+V_{\infty})$ we obtain above, this also gives $\alpha_x(\oX,\oDelta+V_{\infty}) = \min\{1,\alpha(V,\Delta_V)\}$.
\end{proof}

We now relate the volume ratio with the $\alpha$-invariant of the orbifold base.

\begin{lem} \label{lem-compare theta and alpha}
There exists some constant $c>0$ depending only on the dimension such that for any quasi-regular polarized log Fano cone singularity $x\in (X,\Delta;\xi)$ of dimension $n$ with orbifold base $(V,\Delta_V)$, we have
\[
c\cdot \alpha(V,\Delta_V)\ge \Theta(X,\Delta;\xi)\ge  \min\{1,\alpha(V,\Delta_V)\}^n.
\]
\end{lem}

\begin{proof}
This essentially follows from \cite[Lemma 3.12 and Remark 3.13]{Z-mld^K-2}. We provide a (slightly different) proof for the reader's convenience. Let $(\oX,\oDelta+V_{\infty})$ be the associated projective orbifold cone as before. Let $D_V\sim_\bR -(K_V+\Delta_V)$ be an effective $\bR$-divisor and let $D$ be the closure in $X$ of its pullback to $X\setminus\{x\}$. Then $\wt_\xi (D) = A_{X,\Delta}(\wt_\xi)$. The uniform Izumi inequality in \cite[Lemma 3.4]{Z-mld^K-1} thus implies that 
\[\lct_x (X,\Delta;D)\ge c_0\cdot \Theta(X,\Delta;\xi)\] 
for some constant $c_0=c_0(n)>0$. But we also have $\lct(V,\Delta_V;D_V)\ge \lct_x(X,\Delta;D)$ as in the proof of Lemma \ref{lem:compare alpha at vertex and infinity}. As $D_V$ is arbitrary, this gives the first inequality with $c=c_0^{-1}$.

By Lemma \ref{lem:compare alpha at vertex and infinity} and the following Lemma \ref{lem:alpha-vol inequality}, we have
\begin{eqnarray*}
\hvol(x,X,\Delta) &\ge & \alpha_x(\oX,\oDelta+V_{\infty})^n\cdot \vol(-(K_{\oX}+\oDelta+V_{\infty})) \\
                  &  =& \min\{1,\alpha(V,\Delta_V)\}^n \cdot \vol(-(K_{\oX}+\oDelta+V_{\infty})) \, .\\               
\end{eqnarray*}
On the other hand, we have $\hvol_{X,\Delta}(\wt_\xi)=\vol(-(K_{\oX}+\oDelta+V_{\infty}))$ by Lemma \ref{lem:Sasaki vol as global vol}. This gives the second inequality.
\end{proof}

We have used the following statement, which is well-known to experts.

\begin{lem} \label{lem:alpha-vol inequality}
Let $(X,\Delta)$ be a pair of dimension $n$ that is klt at a closed point $x$, and let $L$ be a big $\bR$-Cartier $\bR$-divisor on $X$. Then we have
\[
\hvol(x,X,\Delta)\ge \alpha_x(X,\Delta;L)^n \cdot \vol(L).
\]
\end{lem}

\begin{proof}
Let $t=\alpha(X,\Delta;L)$. Suppose that $x\in X$ is a smooth point and $\vol(L)>\frac{n^n}{t^n}$. Then it is well-known, by a simple dimension count, that there exists some effective $\bQ$-divisor $D\sim_\bQ L$ such that $\mult_x D>\frac{n}{t}$; in particular, 
\[
\alpha_x(X,\Delta;L)\le \lct_x(\oX,\oDelta;D)<t\,,
\]a contradiction. We can apply the same dimension counting argument at a singular point $x\in X$, as long as we replace $\mult_x$ by the minimizing valuation of the normalized volume function, and $n^n$ by the local volume $\hvol(x,X,\Delta)$. 
\end{proof}

\subsection{Bounded family} \label{ss:bdd defn}

In this subsection we define various notions of boundedness.

\begin{defn}
We call $(\cX,\cD)\to B$ a family of pairs if $\cX$ is flat over $B$, the fibers $\cX_b$ are connected, normal and not contained in $\Supp(\cD)$.

We call $B\subseteq (\cX,\cD)\to B$ an $\bR$-Gorenstein family of klt singularities (over a normal but possibly disconnected base $B$) if
\begin{enumerate}
    \item $(\cX,\cD)\to B$ is a family of pairs, and $B\subseteq \cX$ is a section,
    \item $K_{\cX/B}+\cD$ is $\bR$-Cartier and $b\in (\cX_b,\cD_b)$ is a klt singularity for all $b\in B$.
\end{enumerate}
\end{defn}

\begin{defn}
We say that a set $\cC$ of sub-pairs is bounded if there exists a family $(\cX,\cD)\to B$ of pairs over a finite type base $B$ such that for any $(X,D)\in \cC$, there exists a closed point $b\in B$ and an isomorphism $(X,\Supp(D))\cong (\cX_b,\Supp(\cD_b))$. 
\end{defn}

\begin{defn}[{\cite[Definition 2.16]{Z-mld^K-2}}]
We say that a set $\cS$ of polarized log Fano cone singularities is bounded if there exists finitely many $\bR$-Gorenstein families $B_i\subseteq (\cX_i,\cD_i)\to B_i$ of klt singularities over finite type bases, each with a fiberwise good $\bT_i$-action for some nontrivial algebraic torus $\bT_i$, such that every $x\in (X,\Delta;\xi)$ in $\cS$ is isomorphic to $b\in (\cX_{i,b},\cD_{i,b};\xi_b)$ for some $i$, some $b\in B_i$ and some $\xi_b\in N(\bT_i)_\bR$.
\end{defn}

A priori, it may happen that a set of log Fano cone singularities is bounded as a set of sub-pairs, but becomes unbounded when we take into account the log Fano cone structure. Nonetheless, the two boundedness notions coincide if the volume ratios are bounded away from zero and the coefficients belong to a fixed finite set.

\begin{lem} \label{lem:two defn of bdd agree}
Let $\theta>0$ and let $I\subseteq [0,1]$ be a finite set. Let $\cS$ be a set of polarized log Fano cone singularities $x\in (X,\Delta;\xi)$ with coefficients in 
$I$ and $\Theta(X,\Delta;\xi)\ge \theta$. Assume that the underlying set of pairs is bounded. Then $\cS$ is bounded.
\end{lem}

\begin{proof}
Let $f\colon (\cX,\cD)\to B$ be a family of pairs over a finite type base such that for any $x\in (X,\Delta;\xi)$ in $\cS$, we have an isomorphism $(X,\Supp(\Delta))\cong (\cX_b,\Supp(\cD_b))$ for some $b\in B$. After base change along $\cX\to B$ and possibly stratifying $B$, we may assume that $f$ admits a section $\sigma\colon B\to \cX$ so that the above isomorphism induces an isomorphism
\[
\left(x\in (X,\Supp(\Delta))\right) \cong \left(\sigma(b)\in (\cX_b,\Supp(\cD_b))\right).
\]
Since the coefficients belong to the finite set $I$, we may also assign coefficients to $\cD$ and assume that $(X,\Delta)\cong (\cX_b,\cD_b)$. After these reductions, by \cite[Lemma 4.44]{Kol-moduli-book} (or rather its proof) and inversion of adjunction, we know that there exists a finite collection of locally closed subset $B_i$ of $B$ such that the family $(\cX,\cD)$ becomes $\bR$-Gorenstein after base change to $\sqcup_i B_i$ and enumerates exactly all the klt fibers of $B\subseteq (\cX,\cD)\to B$ ({\it cf.} the last part of the proof of \cite[Theorem 3.1]{Z-mld^K-2}). Thus by replacing $B$ with $\sqcup_i B_i$, we may assume that $B\subseteq (\cX,\cD)\to B$ is an $\bR$-Gorenstein family of klt singularities to begin with. Since $\Theta(X,\Delta;\xi)\ge \theta$ by assumption, together with \cite[Lemma 2.15 and Theorem 3.1]{Z-mld^K-2}, we then see that the set $\cS$ is bounded as a set of log Fano cone singularities. 
\end{proof}

We also recall the definition of log birational boundedness. For an $\bR$-divisor $G$, we denote its positive part by $G^+$ and negative part by $G^-$, i.e. $G=G^{+}-G^{-}$ where $G^{+}$, $G^{-}$ are effective without common components.

\begin{defn} \label{defn:birational dominate pairs}
Let $(X,G)$ and $(Y,\Sigma)$ be projective sub-pairs. We say that $(Y,\Sigma)$ log birationally dominates $(X,G)$ if there exist a birational map $\varphi\colon Y\dashrightarrow X$ such that $\Supp(\Sigma)$ contains the birational transform of $\Supp(G)$ and the exceptional divisors of $\varphi$, i.e. ${\rm Supp}(\Sigma)\supseteq{\rm Supp}(\varphi^{-1}_* G)+{\rm Ex}(\varphi)$. We say that $(Y,\Sigma)$ log birationally dominates $(X,G)$ effectively if in addition the $\varphi^{-1}$-exceptional divisors are contained in $\Supp(G^-)$. 

We will say $(Y,\Sigma)$ log birationally dominates $(X,G)$ (effectively) through $\varphi$ if we want to specify the birational map $\varphi\colon X\dashrightarrow Y$.
\end{defn}

Note that if $(Y,\Sigma)$ log birationally dominates $(X,G)$ with $G$ being $\bR$-Cartier, and $G'$ is the birational pullback of $G$, then $\Supp(G')\subseteq \Supp(\Sigma)$.

\begin{defn} \label{defn:birational dominate sets}
Let $\cC$ be a set of projective sub-pairs and let $\cP$ be a set of projective pairs. We say that $\cP$ log birationally dominates $\cC$ (resp. log birationally dominates $\cC$ effectively) if any $(X,G)\in \cC$ is log birationally (resp. log birationally and effectively) dominated by some $(Y,\Sigma)\in\cP$. 

We say that $\cC$ is log birationally bounded if there exists a bounded set $\cP$ of pairs that log birationally dominates $\cC$ ({\it cf.} \cite[Definition 2.4.1]{HMX-BirAut}).
\end{defn}

The following criterion for log birational boundedness is a special case of \cite[Lemma 3.2]{HMX-BirAut} or \cite[Proposition 4.4]{Birkar-bab-1}.

\begin{prop} \label{prop:Birkar's log bir bdd prop}
Let $n$ be a positive integer and let $c_0,c_1>0$. Let $\cC$ be the set of pairs $(X,\Delta+\Gamma)$ of dimension $n$ such that
\begin{itemize}
    \item $-(K_X+\Delta)$ is ample,
    \item the non-zero coefficients of $\Delta$ are at least $c_0$,
    \item $\Gamma$ is a $\bQ$-Cartier, effective, nef Weil divisor,
    \item $|\Gamma|$ defines a birational map and $\vol(\Gamma)\le c_1$.
\end{itemize} 
Then $\cC$ is log birationally bounded. More precisely, there exists a bounded set $\cP$ of projective log smooth pairs $(Y,\Sigma)$ depending only on $n,c_0,c_1$ such that the following are satisfied: For any $(X,\Delta+\Gamma)\in \cC$, there exist some log smooth pair $(Y,\Sigma)\in \cP$ and a birational map $\varphi\colon Y\dashrightarrow X$ such that:
\begin{enumerate}
    \item $(Y,\Sigma)$ log birationally dominates $(X,\Delta+\Gamma)$ through $\varphi$.
    \item There exists some effective and big Cartier divisor $A\le \Sigma$ on $Y$ such that $|A|$ is base point free and $|\Gamma-(\varphi^{-1})^*A|\neq \emptyset$.
\end{enumerate} 
\end{prop}

\begin{proof}
Log birational boundedness follows from \cite[Proposition 4.4(1)]{Birkar-bab-1}, which also gives the property (1). Property (2) follows from \cite[Proposition 4.4(3)]{Birkar-bab-1} (or from the construction of the bounded set $\cP$ in {\it loc. cit.}, as $A$ is simply the birational transform of the movable part of $|\Gamma|$).
\end{proof}

\section{Boundedness}

In this section, we give the proof of our main theorems. The main statement is Theorem \ref{thm:polarized cone bdd}, and we divide its proof into three parts, as outlined in \ref{par:strategy}.

\subsection{Log birational boundedness} \label{ss:log bir bdd}

To prove Theorem \ref{thm:polarized cone bdd}, we first aim to show that the log Fano cone singularities have log birationally bounded projective orbifold cone compactifications. From Section \ref{ss:orbifold}, we have seen that the local alpha invariants of the projective orbifold cones are bounded from below away from the divisor at infinity, and their volumes are also bounded. The situation is thus somewhat similar to those of \cite{Jia-Kss-Fano-bdd}. Our first step is to refine some of the arguments in \cite{Jia-Kss-Fano-bdd} to prove an effective birationality result. Log birationally boundedness is then an immediate consequence.

In the global (Fano) setting, \cite{Jia-Kss-Fano-bdd} proceeds as follow. In order to show that $|-mK_X|$ defines a birational map for some fixed integer $m$, one aims to create isolated non-klt centers on the Fano variety $X$. The main observation from \cite{Jia-Kss-Fano-bdd} is that if both the alpha invariant and the volume are bounded from below, then the volumes of any covering family of subvarieties are also bounded from below, and this allows one to cut down the dimension of the non-klt centers. The next two lemmas show that this strategy still work if we replace the global alpha invariant by the local one.

\begin{lem}[{\it cf.} {\cite[Lemma 3.1]{Jia-Kss-Fano-bdd}}] \label{lem:vol bound of subvar}
Let $X$ be a normal projective variety of dimension $n$ and $L$ a big $\bR$-Cartier $\bR$-divisor on $X$. Let $f \colon Y \to T$ be a projective morphism and $\mu\colon Y \to X$ a surjective morphism. Assume that a general fiber $F$ of $f$ is of dimension $k$ and is mapped birationally onto its image $G$ in $X$. Then for any general smooth point $x\in X$, we have
\[
\vol(L|_G)\ge \frac{\alpha_x(L)^{n-k}}{\binom{n}{k}(n-k)^{n-k}} \vol(L).
\]
\end{lem}

\begin{proof}
This follows from \cite[Lemma 3.1]{Jia-Kss-Fano-bdd} with some small modifications. We sketch the argument for the reader's convenience. By perturbing the coefficients and rescaling, we may assume that $L$ is Cartier. Replacing $f$ by its Stein factorization, we may assume that $F$ is connected. We also assume that $Y$ and $T$ are smooth by taking log resolution. Moreover, by the Bertini Theorem we may replace $T$ by a general complete intersection subvariety and assume that $\mu$ is generically finite. In particular, it is \'etale at the generic point of $F$ (since $F$ is a general fiber). We may also choose $F$ so that $x\in G$. Clearly it suffices to consider the case when $k<n$.

Let $t=f(F)\in T$ and $l\in \bQ_+$. By a direct calculation (using that $F$ has trivial normal bundle in $Y$), we have 
\[
h^0(Y,\mu^*L^{\otimes m}\otimes \cO_Y/\cI_F^{lm}) \le h^0(F,\mu^*L^{\otimes m})\cdot h^0(\cO_T/\fm_t^{lm}) + O(m^{n-1})
\]
for sufficiently large and divisible integers $m$. Hence if
\begin{equation} \label{eq:compare leading term}
    \frac{\vol(L)}{n!}> \frac{\vol(L|_G) \cdot l^{n-k}}{k!\cdot (n-k)!},
\end{equation}
then $h^0(X,mL)>h^0(Y,\mu^*L^{\otimes m}\otimes \cO_Y/\cI_F^{lm})$ for $m\gg 0$. It follows that there exists some effective divisor $D\sim_\bQ L$ such that $\mult_F (\mu^*D)\ge l$; as $\mu$ is \'etale at the generic point of $F$, this also implies that $\mult_G D\ge l$ and therefore 
\[
\alpha_x(L)\le \lct_x(D)\le \frac{n-k}{l}
\]
as $G$ has codimension $n-k$ in $X$. This holds for every $l$ that satisfies \eqref{eq:compare leading term}; the lemma then follows.
\end{proof}

\begin{lem}[{\it cf.} {\cite[Theorem 1.5]{Jia-Kss-Fano-bdd}}] \label{lem-Jiang's lemma}
Let $\varepsilon,\alpha>0$. Let $X$ be a normal projective variety of dimension $n$, and let $L$ be an ample $\bQ$-Cartier $\bQ$-divisor on $X$ such that $(L^n)\ge \varepsilon$ and $\alpha_x(L)\ge \alpha$ for a general point $x\in X$. Then there exists some positive integer $m_0=m_0(n,\varepsilon,\alpha)$ such that $|K_X+\rup{mL}|$ defines a birational map for all $m\ge m_0$.
\end{lem}

\begin{proof}
The assumptions and Lemma \ref{lem:vol bound of subvar} imply that there exists some $m_0=m_0(n,\varepsilon,\alpha)>0$ such that $\vol(mL|_G)>(2k)^k$ (where $k=\dim G$) for any general member $G$ of a covering family of positive dimensional subvarieties of $X$ and all $m\ge m_0$. The argument in \cite[2.31(2)]{Birkar-bab-1} implies that $mL$ is potentially birational (\cite[Definition 3.5.3]{HMX-ACC}), and then the lemma follows from \cite[Lemma 2.3.4]{HMX-BirAut}.
\end{proof}

We can now prove the effective birationality of the orbifold cone compactifications. 

\begin{prop} \label{prop:effective birationality}
Fix some positive integer $n$, a finite coefficient set $I\subseteq [0,1]\cap \bQ$, and some positive real numbers $\varepsilon,\theta>0$.
Then there exist some positive integer $m=m(n,\varepsilon,\theta,I)$ such that $m\cdot I\subseteq \bN$ and for any quasi-regular polarized log Fano cone singularity $x\in (X,\Delta;\xi)$ with
\begin{equation} \label{eq:assumptions on log Fano cone}
    \dim X = n,\mbox{\ \ }\Coef(\Delta)\subseteq I,\mbox{\ \ }\hvol(X,\Delta;\xi)\ge \varepsilon \mbox{\ \ and \ \ }\Theta(X,\Delta;\xi)\ge \theta,
\end{equation}
the following statements hold for its orbifold cone compactification $(\oX,\oDelta+V_\infty)$:
\begin{enumerate}
    \item The pair $(\oX,\oDelta+V_\infty)$ has an $m$-complement.
    \item The linear system $|-m(K_{\oX}+\oDelta+V_\infty)|$ defines a birational map.
\end{enumerate}
\end{prop}

Here we define an $m$-complement of a pair $(X,D)$ as an effective $\bQ$-divisor $\Gamma$ such that $(X,D+\Gamma)$ is log canonical and $m(K_X+D+\Gamma)\sim 0$.

\begin{proof}
Item (1) is the boundedness of complements proved in \cite[Theorem 1.7]{Birkar-bab-1}. Let us prove (2). Let $L=-(K_{\oX}+\oDelta+V_\infty)$. By Lemma \ref{lem:Sasaki vol as global vol} and our assumption on the local volume, we have $\vol(L)\ge \varepsilon$. By Lemmas \ref{lem:compare alpha at vertex and infinity} and \ref{lem-compare theta and alpha}, there exists some positive number $\alpha=\alpha(n,\theta)>0$ such that $\alpha_{x_1}(L)=\alpha_{x_1}(\oX,\oDelta+V_\infty)\ge \alpha$ for all $x_1\in X$. Thus Lemma \ref{lem-Jiang's lemma} guarantees the existence of some positive integer $m=m(n,\varepsilon,\theta,I)$ such that $m\Delta$ has integer coefficients and 
\[
|K_{\oX}+\rup{(m+1)L}| = |mL-V_{\infty}|
\]
defines a birational map. It follows that $|mL|$ also defines a birational map. By taking common multiples, we get a positive integer $m=m(n,\varepsilon,\theta,I)>0$ such that (1) and (2) simultaneously hold. 
\end{proof}

From Lemmas \ref{lem:Sasaki vol as global vol}, we know that 
\[
\vol(M)\le m^n\vol(-(K_{\oX}+\oDelta+V_\infty))\le (mn)^n \theta^{-1}.
\]
Thus by Proposition \ref{prop:Birkar's log bir bdd prop}, this immediately implies that the set of orbifold cone compactifications of quasi-regular log Fano cone singularities satisfying \eqref{eq:assumptions on log Fano cone} is log birationally bounded. Choose some (log bounded) birational model $\varphi\colon(Y,\Sigma)\dashrightarrow \oX$ of the projective orbifold cone $\oX$. Our next task is to reconstruct $X$ from $Y$.

\subsection{Boundedness in codimension one}
\label{ss:bdd in codim 1}

To reconstruct $X$, we need to first understand the exceptional divisors of the birational map $\varphi^{-1}\colon \oX\dashrightarrow Y$. Note that the infinity divisor $V_\infty$ is typically $\varphi^{-1}$-exceptional, and since it depends on the choice of the Reeb vector $\xi$, we will not have much control over it. The next result shows that other than $V_\infty$, the remaining $\varphi^{-1}$-exceptional divisors are essentially ``bounded''. To state it precisely let us make one more definition.

\begin{defn}
Let $(X,\Delta)$ be a pair and let $N$ be a positive integer. A sub-klt $N$-complement of $(X,\Delta)$ is a (not necessarily effective) $\bQ$-divisor $G$ on $X$ such that $N(K_X+\Delta+G)\sim 0$ and $(X,\Delta+G)$ is sub-klt.
\end{defn}

\begin{prop} \label{prop:bdd in codim 1}
Fix some positive integer $n$, a finite coefficient set $I\subseteq [0,1]\cap \bQ$, and some positive real numbers $\varepsilon,\theta>0$. There exist a bounded set $\cP$ of projective log smooth pairs $(Y,\Sigma)$ and a positive integer $N=N(n,\varepsilon,\theta,I)$, such that the following holds for any quasi-regular polarized log Fano cone singularity $x\in (X,\Delta;\xi)$ satisfying \eqref{eq:assumptions on log Fano cone}: 
\begin{enumerate}
    \item $(\oX,\oDelta)$ has a sub-klt $N$-complement $G$ such that $\Supp(G^{-})\subseteq V_\infty\subseteq \Supp(G)$.
    \item There exists some $(Y,\Sigma)\in \cP$ that log birationally dominates $(\oX,\oDelta+G)$ effectively $($Definition \ref{defn:birational dominate pairs}$)$.
\end{enumerate}
\end{prop}

Informally, the implication (2) means that the log Fano cone singularities are bounded in codimension one. The existence of a sub-klt bounded complement will later be used to ensure that certain MMPs exist and terminate.

The proof of the proposition is inspired by the proof of \cite[Theorem 1.6]{HMX-ACC} and \cite{Birkar-bab-1}*{Proposition 7.13}. The main technical part is to construct a bounded sub-klt complement of $(\oX,\oDelta)$ satisfying certain conditions. We first discuss how the existence of such a complement affects boundedness in codimension one.

\begin{lem} \label{lem:sub-klt N-comp imply bdd in codim 1}
Let $N$ be a positive integer and let $\cC$ be a set of projective sub-klt sub-pairs $(X,G)$ satisfying $N(K_X+G)\sim 0$. Assume that $\cC$ is log birationally bounded. Then there exists a bounded set $\cP$ of projective log smooth pairs that log birationally dominates $\cC$ effectively $($Definition \ref{defn:birational dominate sets}$)$.
\end{lem}

\begin{proof}
By assumption, we may choose a bounded set $\cP$ of projective pairs that log birationally dominates $\cC$. After passing to a log resolution, we may assume that $\cP$ is bounded set of log smooth pairs. For any $(X,G)\in \cC$, let $(Y,\Sigma)\in \cP$ be a log smooth pair that log birationally dominates $(X,G)$ through a birational map $\varphi\colon Y\dashrightarrow X$. Write $\varphi^*(K_X+G)=K_Y+G_Y$. Then $G_Y$ is a sub-klt $N$-complement of $Y$ supported on $\Sigma$ (see the remark after Definition \ref{defn:birational dominate pairs}). In particular, $G_Y\le (1-\frac{1}{N})\Sigma$. The discrepancy of any $\varphi^{-1}$-exceptional divisor $F$ must satisfy 
\[
a(F;Y,(1-\frac{1}{N})\Sigma)\le a(F;Y,G_Y)=a(F;X,G)\le 0,
\]
unless $F\subseteq \Supp(G^-)$. Since $(Y,\Sigma)$ is log smooth, the pair $(Y,(1-\frac{1}{N})\Sigma)$ is klt, hence there are only finitely many exceptional divisors with discrepancy at most $0$, and these can be extracted via successive blowups along the strata of $\Sigma$. In other words, up to replacing the bounded set $\cP$ of log smooth pairs, we may assume that the only $\varphi^{-1}$-exceptional divisors are among the components of $G^-$, thus $\cP$ also dominates $\cC$ effectively. 
\end{proof}

We next construct the sub-klt bounded complements on the projective orbifold cones.

\begin{lem} \label{lem:sub-klt N-comp exists}
There exists a positive integer $N=N(n,\varepsilon,\theta,I)$ such that for any quasi-regular polarized log Fano cone singularity $x\in (X,\Delta;\xi)$ satisfying \eqref{eq:assumptions on log Fano cone}, there exists a sub-klt $N$-complement $G$ of $(\oX,\oDelta+V_\infty)$ such that
\begin{enumerate}
\item $\Supp(G^{-})\subseteq V_\infty\subseteq \Supp(G+V_\infty)$, and
\item $-(K_{\oX}+\oDelta+V_{\infty})-\frac{1}{2}G^{+}$ is ample.
\end{enumerate}
\end{lem}

\begin{proof}
We follow the argument of \cite{Birkar-bab-1}*{Proposition 7.13}. Let $m=m(n,\varepsilon,\theta,I)>0$ be the integer given by Proposition \ref{prop:effective birationality}. In particular, there exists an $m$-complement $\Gamma\in \frac{1}{m}|M|$ where $M=-m(K_{\oX}+\oDelta+V_{\infty})$. By Proposition \ref{prop:Birkar's log bir bdd prop} as in the remark right after Proposition \ref{prop:effective birationality}, we find a bounded set $\cP$ of projective log smooth pairs $(Y,\Sigma)$ depending only on $n,\varepsilon,\theta,I$, such that for any $x\in (X,\Delta;\xi)$ satisfying \eqref{eq:assumptions on log Fano cone} and any $m$-complement $\Gamma$ as above, there exists some log smooth pair $(Y,\Sigma)\in \cP$ and some birational map $\varphi\colon \oX\dashrightarrow Y$ such that:
\begin{enumerate}
    \item $(Y,\Sigma)$ log birationally dominates $(\oX,\oDelta+V_\infty+\Gamma)$ through $\varphi$.
    \item There exists some effective and big Cartier divisor $A\le \Sigma$ on $Y$ such that $|A|$ is base point free and $|M-(\varphi^{-1})^*A|\neq \emptyset$. 
\end{enumerate}
Define $\Gamma_Y$ by the crepant pullback formula 
\[
K_Y+\Gamma_Y = \varphi^*(K_{\oX}+\oDelta+V_{\infty}+\Gamma)\sim_\bQ 0.
\]
Since $(Y,\Sigma)$ belongs to a bounded family, We can also choose some positive integer $m_0$ depending only on $\cP$, and some $\bQ$-divisor $B=B^+ - B^-$ in a bounded family where 
\[
B^+\in |A|\quad \mathrm{and} \quad m_0 B^-\in |m_0 A|,
\]
such that $B^+$ is in a general position (by Bertini theorem) and $\Sigma\subseteq \Supp(B^-)$ (this is possible since $A$ is big). By construction, $\Supp(\Gamma_Y)\subseteq \Sigma$ and $(Y,\Gamma_Y)$ is sub-lc. Hence the pair $(Y,\Gamma_Y+B)$ is sub-klt and $K_Y+\Gamma_Y+B\sim_\bQ 0$. Its crepant pullback to $\oX$ is $(\oX,\oDelta+V_{\infty}+\Gamma+B_X)$ where 
\[
B_X^+=(\varphi^{-1})^*B^+\mbox{\ \  and \ \ }B_X^-=(\varphi^{-1})^*B^-\,.
\]
In particular, as $B^+$ and $m_0 B^-$ are both Cartier, the coefficients of $B_X$ belongs to $\frac{1}{m_0}\bZ$.

Choose some $R\in |M-(\varphi^{-1})^*A|$. We may write 
\[
B_X^{-}+R = \lambda V_\infty+mC
\]
where $V_\infty \not\subseteq \Supp(C)$. Note that $B_X^-\sim_\bQ (\varphi^{-1})^*A$ and thus
\[
B_X^{-}+R\sim_\bQ M=-m(K_{\oX}+\oDelta+V_{\infty})\,,
\]
hence $C\sim_\bQ -\mu (K_{\oX}+\oDelta+V_{\infty})$ for some $\mu\le 1$. We also note that the coefficients of $C$ are contained in $\frac{1}{m m_0}\bZ$.

By Lemma \ref{lem-compare theta and alpha}, the orbifold base satisfies $\alpha(V,\Delta_V)\ge \alpha_0$ for some positive constant $\alpha_0=\alpha_0(n,\varepsilon,\theta,I)>0$. We may assume that $\alpha_0<1$. By adjunction, this implies that $(\oX,\oDelta+V_\infty+\alpha_0 C)$ is log canonical in a neighbourhood of $V_\infty$. By Lemma \ref{lem:compare alpha at vertex and infinity}, we also know that $(\oX,\oDelta+V_\infty+\alpha_0 C)$ is log canonical away from $V_\infty$. Hence the pair $(\oX,\oDelta+V_\infty+\alpha_0 C)$ is log canonical everywhere. By \cite{Birkar-bab-1}*{Theorem 1.7}, it has an $N$-complement $C'\ge 0$ for some positive integer $N$ that only depends on the dimension and the coefficients; tracing through the construction above, this in turn means that $N$ only depends on $n,\varepsilon,\theta$ and the finite set $I$. Now consider the linear combination
\[
G:=t(\Gamma+B_X)+(1-t)(\alpha_0 C+C').
\]
for some fixed rational number $t\in (0,1)$ such that $mt\le (1-t)\alpha_0<1$ and $\frac{m_0}{t}\not\in \bZ$. As $\mult_{V_{\infty}}B_X\in \frac{1}{m_0}\bZ$ and $V_\infty \not\subseteq \Supp(\Gamma+C+C')$, the second condition on $t$ simply guarantees that $\mult_{V_{\infty}} G\neq -1$ and hence $V_\infty \subseteq \Supp(G+V_\infty)$. Since $G$ is a convex combination of bounded complements of $(\oX,\oDelta+V_\infty)$ and $(X,V_\infty+\Gamma+B_X)$ is sub-klt, we see that $G$ is a sub-klt $N$-complement of $(\oX,\oDelta+V_\infty)$ after possibly enlarging $N$. Moreover, as $mt\le (1-t)\alpha_0$ by our choice 
of $t$ and $B_X^{-}\le mC$ away from $V_\infty$, we have $tB_X^{-}\le (1-t)\alpha_0 C$ away from $V_\infty$ and therefore $\Supp(G^-)\subseteq V_\infty$. In particular, the resulting sub-klt complement $G$ satisfies (1). 

By construction, $G^{-}\le tB_X^{-}$ and $M-B_X^-$ is pseudo-effective. Thus 
\[
-(K_{\oX}+\oDelta+V_{\infty})-G^-\sim_\bQ \left(\frac{1}{m}-t\right)M+t(M-B_X^-)+(tB_X^- -G^-)
\]
is big. But since both $K_{\oX}+\oDelta$ and $G^-$ are proportional to the ample divisor $V_\infty$, this implies the left hand side above is in fact ample. As 
\[
G=G^+-G^-\sim_\bQ -(K_{\oX}+\oDelta+V_{\infty})\,,
\]it follows that $-(K_{\oX}+\oDelta+V_{\infty})-\frac{1}{2}G^+$ is also ample, proving (2).
\end{proof}

We may now return to the proof of Proposition \ref{prop:bdd in codim 1}.

\begin{proof}[Proof of Proposition \ref{prop:bdd in codim 1}]
Let $N=N(n,\varepsilon,\theta,I)$ be the positive integer from Lemma \ref{lem:sub-klt N-comp exists}. Then for any $x\in (X,\Delta;\xi)$ satisfying \eqref{eq:assumptions on log Fano cone}, there exists a sub-klt $N$-complement $G_0$ of $(\oX,\oDelta+V_\infty)$ such that
\begin{enumerate}
\item[(a)] $\Supp(G_0^{-})\subseteq V_\infty\subseteq \Supp(G_0+V_\infty)$, and
\item[(b)] $-(K_{\oX}+\oDelta+V_{\infty})-\frac{1}{2}G_0^{+}$ is ample.
\end{enumerate}
In particular, part (1) of the proposition is satisfied by $G=G_0+V_\infty$. Possibly replacing $N$ by a larger multiple, we may assume, by Proposition \ref{prop:effective birationality}, that $|-N(K_{\oX}+\oDelta+V_{\infty})|$ defines a birational map. It follows that 
\[
|NG_0^+| = |-N(K_{\oX}+\oDelta+V_{\infty})+NG_0^-|
\]
also defines a birational map.

Condition (b) above together with Lemma \ref{lem:Sasaki vol as global vol} implies that
\[
\vol(G_0^+)\le 2^n \vol(-(K_{\oX}+\oDelta+V_{\infty}))\le (2n)^n \theta^{-1}.
\]
By Proposition \ref{prop:Birkar's log bir bdd prop}, we deduce that there exists a bounded set $\cP$ of projective log smooth pairs, such that for any quasi-regular log Fano cone singularity $x\in (X,\Delta;\xi)$ satisfying \eqref{eq:assumptions on log Fano cone}, there exists some $(Y,\Sigma)\in \cP$ that log birationally dominates $(\oX,\oDelta+V_\infty+G_0^+)$. Using condition (a), we see that $(Y,\Sigma)$ also log birationally dominates $(\oX,\oDelta+G)$. In particular, the sub-klt pair $(\oX,\oDelta+G)$ belongs to a log birationally bounded set. But then by Lemma \ref{lem:sub-klt N-comp imply bdd in codim 1}, after possibly replacing the bounded set $\cP$, we may further assume that $(Y,\Sigma)$ log birationally dominates $(\oX,\oDelta+G)$ effectively. This implies part (2) of the proposition. 
\end{proof}

\subsection{From boundedness in codimension one to boundedness}
\label{ss:MMP step}

Finally, we shall recover the log Fano cone singularity $X$ from its (modified) birational model $Y$ given by the previous subsection (it is important to note that we will not attempt to recover the projective orbifold cone $\oX$, which does not belong to a bounded family). The basic strategy is as follows. Since $X$ is affine, it suffices to recover its section ring from $Y$. Using the birational model $Y$, we will identify a big open subset of $X$ with an open subset $U$ of $Y$, and the question is to find the section ring $\Gamma(\cO_U)$. If $D$ is an effective divisor with support $Y\setminus U$, we may try to run a $D$-MMP on $Y$ and construct its ample model $(\oY,\oD)$. Then $U'=\oY\setminus \oD$ is affine since $\oD$ is ample, and $\Gamma(\cO_U)$ is simply the section ring of $U'$.

Turning to more details, we begin with some general setup. Let $X$ be an affine normal variety, let $\oX$ be a normal projective compactification, and let $V_\infty$ be the divisorial part of $\oX\setminus X$ (a typical example is the orbifold cone compactifications we consider in previous sections). Let $\varphi\colon Y\dashrightarrow \oX$ be a birational map with $Y$ proper. Let $\Sigma_0$ be the sum of $\varphi^{-1}_* V_\infty$ and the exceptional divisors of $\varphi$, and let $D$ be an effective $\bQ$-Cartier $\bQ$-divisor on $Y$ such that $\Supp(D)=\Sigma_0$. A birational contraction $g\colon Y\dashrightarrow\oY$ is called an ample model of $D$ if $\oY$ is proper, $g_*D$ is $\bQ$-Cartier ample, and $D\ge g^*g_*D$.

\begin{lem} \label{lem:recover X by ample model}
Assume that all the $\varphi^{-1}$-exceptional divisors are contained in $V_\infty$, and the ample model $g\colon Y\dashrightarrow \oY$ of $D$ exists. Then the composition $\psi=g\circ\varphi^{-1}\colon \oX\dashrightarrow \oY$ induces an isomorphism $X\cong \oY\setminus g_* \Sigma_0$.
\end{lem}

\begin{proof}
Let $\psi\colon \oX\dasharrow \oY$, $U'=\oY\setminus g_* \Sigma_0$ and let $U=Y\setminus \Sigma_0$. Since $g_* D$ is ample, its complement $U'$ is affine. In order to prove the lemma, it suffices to show that $\psi$ induces an isomorphism $\Gamma(\cO_X)\cong \Gamma(\cO_{U'})$. For this we first show that the induced birational map $(\varphi^{-1})|_X\colon X\dashrightarrow U$ is an isomorphism over some big open sets of both $X$ and $U$.

To see this, note that the exceptional divisors of $\varphi$ are contained in $\Sigma_0$, while the exceptional divisors of $\varphi^{-1}$ are contained in $V_\infty$. We also have $\Supp(\varphi_* \Sigma_0)\subseteq V_\infty$ and $\Supp(\varphi^{-1}_* V_\infty)\subseteq \Sigma_0$ by construction. Thus the complement of all the exceptional locus contains a big open subset of both $X$ and $U$, and $\varphi$ is an isomorphism over this open set. In particular, we have $\Gamma(\cO_X)\cong \Gamma(\cO_U)$.

Let $f\colon W\to Y$, $h\colon W\to \oY$ be a common resolution. Since $g_*D=h_*f^*D$ is ample and $D\ge g^*g_*D$ by assumption, we have $f^*D\ge h^*g_*D$ by the negativity lemma. This implies that $\Supp(f^*D)=h^{-1}(\Supp(g_*D))$ and therefore the induced morphism 
\[
W\setminus \Supp(f^*D) \to \oY\setminus \Supp(g_* D)=U'
\]
is proper, hence they have the same global sections. Similarly the morphism 
\[
W\setminus \Supp(f^*D) \to Y\setminus \Supp(D) = U
\]
is proper as well. Thus they induce isomorphisms
\[
\Gamma(\cO_U)\cong \Gamma(\cO_{W\setminus \Supp(f^*D)}) \cong \Gamma(\cO_{U'}).
\]
Combined with the previous established isomorphism $\Gamma(\cO_X)\cong \Gamma(\cO_U)$, this proves that $\psi$ induces an isomorphism $\Gamma(\cO_X)\cong \Gamma(\cO_{U'})$ and hence $X\cong \oY\setminus g_* \Sigma_0$.
\end{proof}

Now we can put things together to prove the main theorems.

\begin{proof}[Proof of Theorem \ref{thm:polarized cone bdd}]
By \cite[Lemma 2.18]{Z-mld^K-2}, after possibly replacing the positive constants $\varepsilon,\theta$ and the finite set $I$, we may assume that $I\subseteq [0,1]\cap \bQ$. By \cite[Lemma 2.11]{Z-mld^K-2}, after perturbing the Reeb vector $\xi$ and decreasing $\varepsilon,\theta$, we may further assume that the all the polarized log Fano cone singularities in $\cS$ are quasi-regular.

By Proposition \ref{prop:bdd in codim 1}, there exist a bounded set $\cP$ of projective log smooth pairs $(Y,\Sigma)$ and a positive integer $N$, depending only on $n,\varepsilon,\theta$ and $I$, such that the following holds for the projective orbifold cone compactification $(\oX,\oDelta+V_\infty)$ of any log Fano cone singularity $x\in (X,\Delta;\xi)$ in $\cS$: 
\begin{enumerate}
    \item $(\oX,\oDelta)$ has a sub-klt $N$-complement $G$ such that $\Supp(G^{-})\subseteq V_\infty\subseteq \Supp(G)$.
    \item There exists some $(Y,\Sigma)\in \cP$ that log birationally dominates $(\oX,\oDelta+G)$ effectively.
\end{enumerate}

Let $\Sigma_0\subseteq Y$ be the sum of the birational transform of $V_\infty$ and the exceptional divisors of the birational map $\varphi\colon Y\dashrightarrow \oX$. Note that $\Sigma_0\subseteq \Sigma$. In order to apply Lemma \ref{lem:recover X by ample model}, let us show that there exists some effective divisor $D$ with $\Supp(D)=\Sigma_0$ such that the ample model of $D$ exists. Once this is achieved, the remaining step is to run a $D$-MMP in the bounded family $\cP$ for some uniform choice of $D$.

Let $(Y,G_Y)$ be the crepant pull back of $(X,\oDelta+G)$. Then $G_Y$ is a sub-klt $N$-complement of $Y$ which satisfies $\Supp(G_Y)\subseteq \Sigma$ and $\Supp(G_Y^{-})\subseteq \Sigma_0$. In particular, the coefficients of $G_Y$ are at most $1-\frac{1}{N}$. 
 
Consider a new boundary divisor $\Gamma$ on $Y$ as follows: if $F$ is a prime divisor on $Y$ but is not a component of $\Sigma_0$, then we set 
\[
\mult_F (\Gamma):=\mult_F(G_Y)\ge 0;
\]
if $F$ is an irreducible component of $\Sigma_0$, then set 
\[
\mult_F (\Gamma):=\max\left\{0,\mult_F(G_Y)+\frac{1}{2N}\right\}.
\]
Let $D:=\Gamma-G_Y$. By construction, both $\bQ$-divisors $D$ and $\Gamma$ are effective, $\Supp(\Gamma)\subseteq \Sigma$, $\Supp(D)=\Sigma_0$, the coefficients of $\Gamma$ are contained in 
\[
\Lambda:=\left\{0,\frac{1}{2N}, \frac{2}{2N},\dots,1-\frac{1}{2N}\right\}
\]
(in particular, as $(Y,\Sigma)$ is log smooth, the pair $(Y,\Gamma)$ is klt), and we have 
\[
K_Y+\Gamma\sim_\bQ (K_Y+\Gamma)-(K_Y+G_Y)\sim_\bQ D.
\]
As $\Supp(\varphi^*V_\infty)\subseteq  \Sigma_0$ and $V_\infty$ is ample on $\oX$, we know that $\Sigma_0$ is big and the same holds for $D$ as $\Supp(D)=\Sigma_0$. In particular, $K_Y+\Gamma$ is big. Thus by \cite[Theorem 1.2]{BCHM}, the ample model $g\colon Y\dashrightarrow \oY$ of $K_Y+\Gamma\sim_\bQ D$ exists. By Lemma \ref{lem:recover X by ample model}, the composition $\psi\colon \oX\dashrightarrow \oY$ induces an isomorphism $X\cong \oY\setminus g_*\Sigma_0$. Since $\Sigma$ contains the birational transform of $\oDelta$ in its support, we also see that $\Supp(\Delta)=(g_*\Sigma_1)|_X$ for some reduced divisor $\Sigma_1\le \Sigma$.

To summarize, we have proved the following. For any log Fano cone singularity $x\in (X,\Delta;\xi)$ in $\cS$, there exist a log smooth pair $(Y,\Sigma)$ from the bounded set $\cP$, two reduced divisors $\Sigma_0,\Sigma_1\le \Sigma$, and an effective divisor $\Gamma$ supported on $\Sigma$ with coefficients in $\Lambda$, such that $K_Y+\Gamma$ is big and 
\[
(X,\Supp(\Delta))\cong (\oY\setminus g_*\Sigma_0, (g_*\Sigma_1)|_{\oY\setminus g_*\Sigma_0}),
\]
where $g\colon Y\dashrightarrow \oY$ is the ample model of $K_Y+\Gamma$.

Since $\cP$ is bounded, all the pairs $(Y,\Sigma)$ in $\cP$ arise as the fibers of some bounded family $(\cY, \Sigma_{\cY})\to B$ of pairs over a finite type base $B$, i.e. there exists $b\in B$ such that 
\[
(Y,\Sigma)\cong (\cY_b, \Sigma_{\cY_b}):=(\cY, \Sigma_{\cY})\times_B b .
\]
After stratifying $B$ and performing a base change, we may assume that $B$ is smooth, $(\cY, \Sigma_{\cY})\to B$ is log smooth, and every irreducible component of $\Sigma$ is the restriction of some component of $\Sigma_\cY$. In particular, there are $\bQ$-divisors $\Sigma_{0,\cY},\Sigma_{1,\cY}$ and $\Gamma_{\cY}$ supported on $\Sigma_{\cY}$ that restricts to $\Sigma_0, \Sigma_1$ and $\Gamma$ on the fiber $(Y,\Sigma)$.

To conclude the proof of the boundedness, we note that by \cite[Theorem 1.8]{HMX-BirAut}, the volume of $K_{\cY_b}+\Gamma_{\cY_b}$ is locally constant in $b\in B$; moreover, over the components of $B$ where $K_{\cY_b}+\Gamma_{\cY_b}$ is big, the relative ample model $h\colon \cY\dashrightarrow \overline{\cY}$ of $K_{\cY}+\Gamma_{\cY}$ over $B$ exists, whose restriction over $b$ yields the ample model $g_b\colon \cY_b\dashrightarrow \overline{\cY}_b$ of $K_{\cY_b}+\Gamma_{\cY_b}$. By Noetherian induction, after possibly stratifying $B$ again, we may assume that the restriction of $h_* \Sigma_{0,\cY}$ (resp. $h_* \Sigma_{1,\cY}$) to the fiber  $\overline{\cY}_b$ is exactly $(g_b)_*\Sigma_{0,\cY_b}$ (resp. $(g_b)_*\Sigma_{1,\cY_b}$). Set 
\[
\cX:=\overline{\cY}\setminus h_* \Sigma_{0,\cY} \quad \mathrm{and}\quad \Delta_{\cX}:=h_*\Sigma_{1,\cY}|_{\cX}.
\]
There are only finitely many choices of $\Sigma_{0,\cY}, \Sigma_{1,\cY}$ and $\Gamma_{\cY}$, as their coefficients belong to the finite set $\Lambda\cup \{1\}$. This leads to finitely many families $(\cX,\Delta_{\cX})\to B$ as above. By the previous discussion, for any log Fano cone singularity $x\in (X,\Delta;\xi)$ in $\cS$, the pair $(X,\Supp(\Delta))$ appears as a fiber of one of the families $(\cX,\Delta_{\cX})\to B$; therefore, the set of pairs underlying $\cS$ is bounded. By Lemma \ref{lem:two defn of bdd agree}, this implies that $\cS$ is also a bounded set of log Fano cone singularities.
\end{proof}

\begin{proof}[Proof of Corollary \ref{c-bounded cone}]
 It suffices to prove $x\in (X,\Delta;\xi)$ satisfies the condition of Theorem \ref{thm:polarized cone bdd}, where $x$ is the vertex of the cone, and $\xi$ corresponds to the $\bG_m$-action given by the cone structure. 

By Lemma \ref{lem-compare theta and alpha}, $\Theta(X,\Delta;\xi)\ge (\min\{\alpha_0,1\})^n$. One can directly calculate
\[
\hvol_{X,\Delta}(\wt_{\xi})=r(-K_V-\Delta_V)^{n-1}\ge \varepsilon 
\]
(see e.g. \cite[Lemma 3.4]{Z-mld^K-2}), which implies that $\hvol(x,X,\Delta)\ge \varepsilon\cdot  (\min\{\alpha_0,1\})^n$. Thus the set $\cS$ is bounded by Theorem \ref{thm:polarized cone bdd}.
\end{proof}

\begin{proof}[Proof of Theorem \ref{thm-Fanocone bounded}]
This directly follows from Theorem \ref{thm:polarized cone bdd} when $\Theta(X,\Delta;\xi)=1$.
\end{proof}

\begin{proof}[Proof of Theorem \ref{thm-discrete volume}]
By Theorem \ref{thm:SDC}, every klt singularity $x\in (X,\Delta)$ has a special degeneration to a K-semistable log Fano cone singularity $x_0\in (X_0,\Delta_0;\xi_v)$ with $\hvol(x,X,\Delta)=\hvol(x_0,X_0,\Delta_0)$. Moreover, the coefficients of $\Delta_0$ belong to the finite set
\[
I^+:=\{\sum_{i}m_ia_i \mid m_i\in\bN, a_i\in I\} \cap [0,1].
\]
Therefore, Theorem \ref{thm-discrete volume} follows from Theorem \ref{thm-Fanocone bounded} and the constructibility of the local volume function in bounded families \cite[Theorem 1.3]{Xu-quasi-monomial} (see also \cite[Theorem 3.5]{HLQ-vol-ACC} for the real coefficient case).
\end{proof}

\begin{rem}
It is important to understand more about the set $\widehat{\rm Vol}_{n,I}$.
When $I=\{0\}$, it is proved in \cite{LX-cubic-3fold} that the maximal number in $\widehat{\rm Vol}_{n}:=\widehat{\rm Vol}_{n,\{0\}}$ is $n^n$, the local volume of a smooth point. It is conjectured that the second largest number in $\widehat{\rm Vol}_{n}$ is $2(n-1)^n$ (the local volume of an ordinary double point), but for now this is known only when $n\le 3$.
\end{rem}

\section{Minimal log discrepancy of Koll\'ar components}

It is observed in \cite{Z-mld^K-1,Z-mld^K-2} that the boundedness of log Fano cone singularities is closely related to the boundedness of $\mldk$, the minimal log discrepancies of Koll\'ar component. The goal of this section is to use the boundedness result from the previous section to prove some upper bounds of $\mldk$ that only depend on the local volume.

\begin{defn}[\cite{Xu-pi_1-finite}]
Let $x\in (X,\Delta)$ be a klt singularity and let $E$ be a prime divisor over $X$. If there exists a proper birational morphism $\pi\colon Y\to X$ such that $E= \pi^{-1}(x)$ is the unique exceptional divisor, $(Y,E+\Delta_Y)$ is plt and  $-(K_Y+\Delta_Y+E)$ is $\pi$-ample, we call $E$ a \emph{Koll\'ar component} over $x\in (X,\Delta)$ and $\pi\colon Y\to X$ the plt blowup of $E$.
\end{defn}

\begin{defn}[\cite{Z-mld^K-1}] \label{defn:mld^K}
Let $x\in (X,\Delta)$ be a klt singularity. The minimal log discrepancy of Koll\'ar components, denoted $\mldk (x,X,\Delta)$, is the infimum of the log discrepancies $A_{X,\Delta}(E)$ as $E$ varies among all Koll\'ar components over $x\in (X,\Delta)$.
\end{defn}

\begin{thm}[{\it cf.} {\cite[Conjecture 1.7]{Z-mld^K-1}}] \label{thm:mld^K bdd}
Let $n\in\bN$, $\varepsilon>0$ and let $I\subseteq [0,1]$ be a finite set. Then there exists some constant $A>0$ depending only on $n,\varepsilon,I$ such that 
\[
\mldk(x,X,\Delta)\le A
\]
for any $n$-dimensional klt singularity $x\in (X,\Delta)$ with ${\rm Coeff}(\Delta)\subseteq I$ and $\hvol(x,X,\Delta)\ge \varepsilon$.
\end{thm}

The idea of the proof is quite straightforward: by the Stable Degeneration Theorem \ref{thm:SDC} and the Boundedness Theorem \ref{thm-Fanocone bounded}, we know that klt singularities with ${\rm Coeff}(\Delta)\subseteq I$ and $\hvol(x,X,\Delta)\ge \varepsilon$ degenerate into a bounded family of log Fano cone singularities $(X_0,\Delta_0;\xi)$. 
Then it suffices to show that there is a uniform choice of Koll\'ar components over this family that deforms through the stable degeneration to the original singularities. We achieve this by analyzing the syzygies of the singularities, and show that there is a uniform rational perturbation $\xi_0$ of $\xi$ such that $(X,\Delta)$ also degenerates to $(X_0,\Delta_0;\xi_0)$, yielded by a Koll\'ar component $E$ over $(X,\Delta)$ by \cite[Theorem 4.1]{XZ-SDC}. We conclude by  the fact that the log discrepancy of $E$ is the same as the one of the Koll\'ar component induced by $(X_0,\Delta_0;\xi_0)$. 

\subsection{Filtered resolution}

Before we prove Theorem \ref{thm:mld^K bdd}, we need to recall a few results on syzygies of filtered algebras. These results should be well known but we haven't found a good reference. Throughout this subsection, let $(S,\fm)$ be a complete local Noetherian $\bk$-algebra endowed with an $\fm$-adic filtration\footnote{All filtrations in this section are exhaustive, decreasing, left-continuous, multiplicative and indexed by $\bR$. Here we say a filtration $\cF$ on $M$ is exhaustive if $\cap_\lambda \cF^\lambda M = 0$ and $\cup_\lambda \cF^\lambda M = M$.} $\cF$, i.e. there exists some positive constants $c_0, c_1$ such that 
\[
\fm^{\lceil c_0 \lambda\rceil}\subseteq \cF^\lambda S \subseteq \fm^{\lfloor c_1 \lambda\rfloor}
\]
for all $\lambda\ge 0$. This implies that $S$ is also complete with respect to the filtration $\cF$. Assume that the associated graded ring $\gr_\cF S$ is finitely generated. 

\begin{lem}
Consider an $S$-module $M$ with a compatible filtration also denoted by $\cF$. Then the following conditions are equivalent:
\begin{enumerate}
    \item $\gr_\cF M$ is finitely generated over $\gr_\cF S$.
    \item $M$ is finitely generated, and there exists some $u_i\in \cF^{\mu_i} M$ ($i=1,\dots,N$) such that $\cF^\lambda M = \sum_{i=1}^N \cF^{\lambda-\mu_i} S\cdot u_i$ for all $\lambda\in \bR$.
\end{enumerate}
\end{lem}

\begin{proof}
Clearly (2) implies (1). For the other direction, note that (1) implies the existence of some $\mu\in \bR$ and some $u_i\in \cF^{\mu_i} M$ ($i=1,\dots,N$) such that $\cF^\mu M = M$ and that 
\[
\cF^\lambda M = \cF^{>\lambda} M + \sum_{i=1}^N \cF^{\lambda-\mu_i} S\cdot u_i
\]
for all $\lambda\in \bR$. By the finite generation of $\gr_\cF S$ and $\gr_\cF M$, we also know that the set $\{\lambda\,|\,\gr^\lambda_\cF M\neq 0\}$ is discrete. Thus as $S$ is complete, (2) follows from a repeated application of the above equality.
\end{proof}

We say that the filtration $\cF$ on $M$ is good if any of the equivalent conditions in the above lemma holds. In particular, if $\cF$ is a good filtration then there exists some $\mu\gg 0$ such that 
\[
\cF^{\lambda+\mu} M = \cF^\lambda S\cdot \cF^\mu M
\]
for all $\lambda\ge 0$. Moreover, $M$ is complete with respect to the good filtration $\cF$; in other words, if $u_k\in \cF^{\lambda_k}M$ ($k=1,2,\dots$) is an infinite sequence such that $\lambda_k\to +\infty$, then the formal series $\sum_{k=1}^{\infty} u_k$ converges in $M$ (this follows from the equivalent condition (2) and the completeness of $S$).

\begin{lem} \label{lem:criterion for filtered exactness}
Let $M'\stackrel{\psi}{\to} M\stackrel{\varphi}{\to} M''$ be a complex of filtered $S$-modules. Assume that the filtration on $M$ is good and the induced complex
\begin{equation} \label{eq:graded cpx exact}
\xymatrix{
\gr(M')\ar[r]^-{\bar{\psi}} & \gr(M) \ar[r]^-{\bar{\varphi}} & \gr(M'')
}
\end{equation}
is exact. Then $M'\stackrel{\psi}{\to} M\stackrel{\varphi}{\to} M''$ is filtered exact, i.e.,
\begin{equation} \label{eq:filtered exact}
\xymatrix{
\cF^\lambda M'\ar[r]^-{\psi} & \cF^\lambda M\ar[r]^-{\varphi} & \cF^\lambda M''
} 
\end{equation}
is exact for all $\lambda\in \bR$.
\end{lem}

\begin{proof}
Since $\gr(M)$ is finitely generated over $\gr_\cF S$ by assumption, we may replace $M'$ by a submodule so that \eqref{eq:graded cpx exact} is still exact and $\gr(M')$ is finitely generated. If \eqref{eq:filtered exact} is filtered exact for this submodule, then the same is true for $M'$. Thus it suffices to prove the lemma when the filtration on $M'$ is also good.

Exactness of \eqref{eq:graded cpx exact} implies that
\[
\ker(\varphi)\cap \cF^\lambda M \subseteq \psi(\cF^\lambda M') + \ker(\varphi)\cap \cF^{>\lambda} M
\]
for any $\lambda \in \bR$. By induction on the filtration index $\lambda$ (which belongs to a discrete set) this in turn implies that for any $u\in \ker(\varphi)\cap \cF^\lambda M$ there exists an infinite sequence $u_k\in \cF^{\lambda_k}M'$ with $\lambda\le \lambda_1 < \lambda_2< \dots$ and $\lambda_k\to \infty$ such that $u-\psi(u')\in \cF^\mu M$ for any $\mu\in \bR$, where $u'=\sum_k u_k \in \cF^\lambda M'$ (the series converges since the filtration on $M'$ is good). Since the filtration on $M$ is assumed to be exhaustive and in particular $\cap_\mu \cF^\mu M = 0$, we deduce that $u=\psi(u')$ and hence \eqref{eq:filtered exact} is filtered exact.
\end{proof}

The following is the main result of this subsection. We denote by $S(\mu)$ the filtered free $S$-module that is isomorphic to $S$ as an $S$-module but with a filtration given by $\cF^\lambda S(\mu) = \cF^{\lambda+\mu} S$. 

\begin{lem} \label{lem:lift resolution}
Let $M$ be an $S$-module with a compatible good filtration and let 
\begin{equation} \label{eq:graded free resolution}
\xymatrix{
\cdots \ar[r]^-{\bphi_2} & \oM_2 \ar[r]^-{\bphi_1} & \oM_1 \ar[r]^-{\bphi_0} & \gr(M) \ar[r] & 0 
}
\end{equation}
be a graded free resolution of $\gr(M)$, i.e. for each $i$ there exists some $\lambda_{ij}\in \bR$ such that $\oM_i\cong \oplus_j \gr_\cF(S)(\lambda_{ij})$. Then it can be lifted to a filtered exact free resolution
\[
\xymatrix{
\cdots \ar[r]^{\varphi_2} & M_2 \ar[r]^{\varphi_1} & M_1 \ar[r]^{\varphi_0} & M \ar[r] & 0
}
\]
of $M$, i.e. the associated graded complex is \eqref{eq:graded free resolution}.
\end{lem}

\begin{proof}
First set $M_1=\oplus_j S(\lambda_{1j})$ and lift $\bar{\varphi}_0$ to a map $\varphi_0\colon M_1\to M$. By Lemma \ref{lem:criterion for filtered exactness}, the exactness of $\oM_1\to \gr(M)\to 0$ implies the filtered exactness of $M_1\to M\to 0$; in particular, $\varphi_0$ is surjective. Let $M'=\ker(\varphi_0)$ and equip it with the filtration induced from $M_1$. We claim that $\gr(M')=\ker(\bar{\varphi_0})$. Indeed, $\gr(M')$ injects into $\gr(M_1)=\oM_1$ as a general property of induced filtrations on submodules. Clearly we also have $\gr(M')\subseteq \ker(\bar{\varphi_0})$. On the other hand, any homogeneous element $\bar{u}\in \ker(\bar{\varphi_0})$ is represented by some $u\in \cF^\lambda M_1$ such that $\varphi_0(u)\in \cF^{>\lambda} M$. As $M_1\to M\to 0$ is filtered exact, this implies that $\varphi_0(u) = \varphi_0(u_1)$ for some $u_1\in \cF^{>\lambda} M_1$. But then $u'=u-u_1\in \ker(\varphi_0)=M'$; moreover, we have $u'\in \cF^\lambda M'$ and it is also a lift of $\bar{u}$. It follows that $\ker(\bar{\varphi_0})\subseteq \gr(M')$. This proves the claim. From here we deduce that $\cdots\to\oM_2\to \gr(M')$ is also a graded free resolution. By construction, the sequence $0\to M'\to M_1\to M\to 0$ is also filtered exact. Thus we may repeat the same argument above with $M'$ in place of $M$ and the lemma follows.
\end{proof}

\subsection{Lifting valuations}

Next consider a singularity $x\in X=\Spec(R)$ and a quasi-monomial valuation $v\in \Val_{X,x}$ such that $R_0:=\gr_v R$ is finitely generated. As in \cite[paragraph after Theorem 4.1]{XZ-SDC}, every $\bR$-divisor $\Delta$ has an induced degeneration $\Delta_0$ on $X_0:=\Spec(R_0)$. The natural grading on $R_0$ induces a Reeb vector field $\xi_v$ on $X_0$ generating a torus $\bT=\langle \xi_v \rangle$. In this setup, using terminologies from K-stability, we say that $x\in (X,\Delta)$ degenerates into $(X_0,\Delta_0;\xi_0)$ via the $\bR$-test configuration induced by $v$. For the application to Theorem \ref{thm:mld^K bdd}, $x\in (X,\Delta)$ will be a klt singularity whose local volume is bounded from below, and $(X_0,\Delta_0;\xi_0)$ is its K-semistable log Fano cone degeneration. Denote by 
\[
\Phi_{\xi_v}(R_0):=\wt_{\xi_v}(R_0\setminus\{0\})=v(R\setminus\{0\})\subseteq \bR_{\ge 0}
\]
the value semigroup which is a discrete set that only depends on the graded ring $R_0$ and the Reeb vector $\xi_v$. We shall identify the Reeb cone $\ft^+_\bR$ with the set of semigroup homomorphisms $\Phi_{\xi_v}(R_0)\setminus \{0\} \to \bR_+$ (in particular, $\xi_v$ is identified with the natural inclusion $\Phi_{\xi_v}(R_0) \subseteq  \bR_{\ge 0}$). The question we need to address is which (rational) perturbation of $\xi_v$ inside the Reeb cone can be lifted to a (divisorial) valuation on $X$, as such a lift will help us control the $\mldk$ of the singularity $x\in X$.

For any finite set $\{\bar{u}_1,\dots,\bar{u}_N\}$ of homogeneous generators of $R_0$, 
we have a graded surjection
\begin{equation} \label{eq:S->R_0}
\bphi_0\colon S:=\bk[x_1,\dots,x_N]\to R_0
\end{equation}
sending $x_i$ to $\bar{u}_i$. Here the grading of $S$ is defined by setting $\deg(x_i)=\deg(\bu_i)$. This also induces a filtration on $\hat{S}:=\bk[\![x_1,\dots,x_N]\!]$ such that $\gr(\hat{S})\cong S$.

\begin{defn} \label{def:syzygy complexity}
Let $c>0$ be a positive number. We say that the syzygy complexity of the graded algebra $R_0$ is at most $c$, if there exist homogeneous generators $\bar{u}_1,\dots,\bar{u}_N$ of $R_0$ (for some $N\le c$) such that $\deg(\bu_i)\le c$ for all $i$, and for the induced map $\bphi_0$ in \eqref{eq:S->R_0}, there exists a graded free resolution of $R_0$ (viewed as a graded $S$-module)
\begin{equation} \label{eq:resolution of R_0}
\xymatrix{
\oM_2\ar[r]^-{\bphi_2} & \oM_1\ar[r]^-{\bphi_1} & S \ar[r]^-{\bphi_0} & R_0 \ar[r] & 0 
}
\end{equation}
such that all the matrix entries of $\bphi_1$ and $\bphi_2$ have degrees at most $c$.
\end{defn}

The technical result of this subsection is that the liftability of the perturbed Reeb vector only depends on the syzygy complexity of $R_0$ and the value semigroup $\Phi_{\xi_v}(R_0)$. To prepare for such a statement, fix some $c>0$ such that the syzygy complexity of $R_0$ is at most $c$, and choose some generators $\bar{u}_i$ and a graded free resolution of $R_0$ as in \eqref{eq:resolution of R_0} that satisfy the conditions in Definition \ref{def:syzygy complexity}. Let $\lambda_i:=\deg(\bu_i)$. By Lemma \ref{lem:lift resolution}, we can lift \eqref{eq:resolution of R_0} to a filtered free resolution
\begin{equation} \label{eq:resolution of R}
\xymatrix{
M_2\ar[r]^-{\varphi_2} & M_1\ar[r]^-{\varphi_1} & \hat{S} \ar[r]^-{\varphi_0} & \hat{R} \ar[r] & 0
}   
\end{equation}
of the $\hat{S}$-module $\hat{R}$ (the completion of $R$). In concrete terms, this means that the matrix entries of the homomorphism $\bphi_i$ are given by the initial terms of the matrix entries of $\varphi_i$. For example, to lift $\bphi_0$ to $\varphi_0$, we choose a lift $u_i\in R$ of $\bu_i$ for each $i$, namely, $u_i\in \fa_{\lambda_i} (v)$ and its reduction in $\fa_{\lambda_i}(v)/\fa_{>\lambda_i}(v)$ is $\bu_i$. We then lift $\bphi_0$ to a map
\[
\varphi_0\colon \bk[x_1,\dots,x_N]\to R
\]
sending $x_i$ to $u_i$, which is surjective after passing to the completion. 

Each Reeb vector $\xi\in \ft^+_\bR$ induces a valuation of $R_0$, and there is a natural candidate for its lift to a valuation of $R$. Indeed, since we identify Reeb vectors with homomorphism $\Phi_{\xi_v}(R_0)\setminus \{0\} \to \bR_+$, there is a natural map 
\[
\iota\colon \ft^+_\bR \to \bR^N_+, \quad \xi\mapsto (\xi(\lambda_1),\dots,\xi(\lambda_N)),
\]
whose image $\iota(\ft_\bR^+)$ is an open subset of the rational envelope $\Lambda\subseteq \bR^N$ of $\lambda=\iota(\xi_v)=(\lambda_1,\dots,\lambda_N)$ (i.e. the smallest rational affine linear subspaces in $\bR^N$ containing $\lambda$). 
For any $\mu\in \bR^N_+$ we get a monomial valuation $\wt_\mu$ on $\hat{S}$ which induces a filtration $\cF_\mu$ on $\hat{S}$. Through the surjection $\varphi_0\colon \hat{S}\to \hat{R}$, it also induces a filtration (also denoted by $\cF_\mu$) on $\hat{R}$. For any $u\in\hat{R}\setminus \{0\}$ set
\begin{equation} \label{eq:candidate lift}
v_\mu(u):=\max\{t\in\bR\,|\,u\in \cF_{\mu}^t \hat{R}\};
\end{equation}
when $\mu=\iota(\xi)$ this is our candidate lift of $\xi$. Note that $v_\mu$ need not be a valuation, and our next task is find conditions that make it into a valuation.

For this we need an auxiliary result. 
For any $\mu\in \bR^N_+$ and any $u\in \hat{S}$, we denote by $\mathrm{in}_\mu (u)$ the initial term of $u$ with respect to $\wt_\mu$. 

\begin{lem} \label{lem:uniform nbh}
For any $u\in \hat{S}$, there exists an open neighbourhood $U\subseteq \ft^+_\bR$ of $\xi_v$ depending only on $\lambda$ and $\wt_\lambda (u)$ such that $\mathrm{in}_{\mu} (u) = \mathrm{in}_\lambda (u)$ for all $\mu\in \iota(U)$.
\end{lem}

\begin{proof}
Every monomial in $\hat{S}$ is of the form $x_1^{a_1}\cdots x_N^{a_N}$ for some $(a_1,\dots,a_N)\in \bN^N$. Observe that (both facts are elementary to verify):
\begin{enumerate}
    \item if $a,b\in \bN^N$ and $\langle \lambda,a \rangle = \langle \lambda,b \rangle$, then for any $\xi\in\ft^+_\bR$ it also holds that $\langle \iota(\xi),a \rangle = \langle \iota(\xi),b \rangle$;
    \item if $a,b\in \bN^N$ is such that $\langle \lambda,a \rangle < \langle \lambda,b \rangle$, then there exists an open neighbourhood $U\subseteq \bR^N$ of $\lambda$ depending only on $a$ and $\lambda$ such that $\langle \mu,a \rangle < \langle \mu,b \rangle$ for all $\mu\in U$.
\end{enumerate}
Note that there are only finitely many $a\in \bN^N$ such that $\langle \lambda,a \rangle = \wt_\lambda (u)$. The lemma then follows by translating these facts into statements about weights of monomials in $\hat{S}$.
\end{proof}

By Lemma \ref{lem:uniform nbh}, we can choose a small neighbourhood $U\subseteq \ft_\bR^+$ of $\xi_v$, depending only on $\lambda$ and the weights of the matrix entries of $\varphi_1$ and $\varphi_2$ (with respect to $\lambda$), such that $\mathrm{in}_\mu (\varphi_i) = \mathrm{in}_\lambda (\varphi_i)$ for all $i=1,2$ and all $\mu\in \iota(U)\subseteq \bR^N_+$. In particular, the first three terms $\oM_2\to \oM_1\to S$ in \eqref{eq:resolution of R_0} is the reduction of $M_2\to M_1\to \hat{S}$ with respect to $\wt_\mu$ for any $\mu\in \iota(U)$.  As there are only finitely many values in the semigroup $\Phi_{\xi_v}(R_0)$ that are bounded from above, we see that in fact $U$ only depends on $\Phi_{\xi_v}(R_0)$ together with the syzygy complexity of $R_0$ (that is, if $R_0$ has syzygy complexity at most $c$, then $U$ can be chosen in terms of $\Phi_{\xi_v}(R_0)$ and $c$). We now show that every $\xi\in U$ lifts to a valuation of $R$ through \eqref{eq:candidate lift}.

\begin{lem} \label{lem:lift valuation}
For any $\xi\in U$, there exists a quasi-monomial valuation $v_\xi\in \Val_{X,x}$ such that $\gr_{v_\xi} R \cong R_0$ and its induced Reeb vector on $X_0$ is $\xi$. 
\end{lem}

\begin{proof}
Let $\mu = \iota(\xi)$. Recall that we have a filtration $\cF_\mu$ of $\hat{R}$ induced by the valuation $\wt_\mu$ on $\hat{S}$. Similarly, we have an induced filtration (still denoted by $\cF_\mu$) on $\hat{I}:=\ker(\varphi_0)\subseteq \hat{S}$. Note that $\hat{I}=\varphi_1 (M_1)$ since \eqref{eq:resolution of R} is exact. As $M_1$, $M_2$ are free $\hat{S}$-modules, there are also natural filtrations (again denoted $\cF_\mu$) on them so that \eqref{eq:resolution of R} is compatible with these new filtrations (on each free summand of $M_i$ the filtration is simply given by an appropriate degree shift of the filtration $\cF_\mu$ of $\hat{S}$). 

As in \eqref{eq:candidate lift}, for $M=M_i$ or $\hat{S}$ and any $u\in M\setminus \{0\}$, we set
$v_\mu(u):=\max\{t\in\bR\,|\,u\in \cF_{\mu}^t M\}$. For ease of notation, we shall write $\gr_\mu (\hat{R})$ instead of $\gr_{\cF_\mu} (\hat{R})$ etc. We claim that 
\begin{equation} \label{eq:gr_mu = R_0}
\gr_\mu \hat{R}\cong R_0.
\end{equation}
Note that the kernel of the induced map $S\cong \grhat{S}\to \grhat{R}$ is $\grhat{I}$. Thus in view of the exact sequence \eqref{eq:resolution of R_0}, to prove the claim it suffices to show that $\oM_1 \cong \gr_\mu (M_1)$ surjects onto $\grhat{I}$. Equivalently, we need to show that the map $\varphi_1\colon M_1\to \hat{S}$ is strict with respect to the filtration $\cF_\mu$, namely, $\varphi_1(\cF_\mu M_1) = \varphi_1 (M_1)\cap \cF_\mu \hat{S} = \cF_\mu \hat{I}$.

To this end, let $u\in \hat{I}$; among its preimages in $M_1$, choose one (denoted by $r$) that lies in the smallest filtered subspaces $\cF_\mu^t M_1$, namely, let $r\in \varphi^{-1}_1 (u)$ be such that $v_\mu (r')\le v_\mu (r)$ for all $r'\in \varphi^{-1}_1 (u)$. This is possible since by our choice of the filtration $\cF_\mu$ on the free $\hat{S}$-module $M_1$, the set $v_\mu(M_1\setminus \{0\})$ of jumping numbers is a discrete set that is bounded from below, and for any $r\in \varphi^{-1}_1 (u)$ we have $v_\mu(u)\ge v_\mu(r)$ as $\varphi_1$ is compatible with the filtration $\cF_\mu$, thus $v_\mu (\varphi^{-1}_1 (u))$ is also bounded from above. Let us prove that $v_\mu(u)=v_\mu(r)$. Suppose not, then $v_\mu(u)> v_\mu(r)$, hence the reduction $\bar{r}\in \oM_1$ of $r\in M_1$ is in the kernel of $\bphi_1$, hence by the exactness of \eqref{eq:resolution of R_0} we see that $\bar{r}=\bphi_2 (\bar{s})$ for some $\bar{s}\in \oM_2$. Choose some $s\in M_2$ whose reduction is $\bar{s}$. Then we have $v_\mu (r-\varphi_2(s))>v_\mu (r)$. But since \eqref{eq:resolution of R} is exact, we also have $r-\varphi_2(s)\in \varphi^{-1}_1(u)$. This contradicts our initial choice of $r$. Thus the equality $v_\mu(u)=v_\mu(r)$ holds and we conclude that $\varphi_1(\cF_\mu M_1) = \cF_\mu \hat{I}$. As discussed at the beginning of the proof, this proves the claim \eqref{eq:gr_mu = R_0}. 

The lemma is now a formal consequence of \eqref{eq:gr_mu = R_0}. Since $R_0$ is an integral domain, we deduce that $v_\mu$ is a valuation (see e.g. \cite[Lemma 2.4]{Z-SDC-survey}; it is quasi-monomial since $R_0$ is finitely generated (see e.g. \cite[Discussion 3.16]{ELS03}). By construction, we have $v_\mu (u)>0$ for all $u\in \fm\subseteq \hat{R}$. Hence the restriction of $v_\mu$ to $R$ gives the desired valuation.
\end{proof}

Specializing to the stable degeneration setting, we obtain the following.

\begin{cor} \label{cor:lift kc}
Let $c_1,c_2>0$ be two positive constants, let $m$ be a positive integer and let $x_0\in (X_0=\Spec(R_0),\Delta_0;\xi_0)$ be a log Fano cone singularity. Assume that $m\Delta_0$ is Cartier, $R_0$ has syzygy complexity at most $c_1$, and $\wt_{\xi_0}(m\Delta)\le c_2$. Then there exists a rational perturbation $\xi_1\in \ft^+_\bR$ of the Reeb vector $\xi_0$ depending only on $c_1,c_2,m$ and the value semigroup $\Phi_{\xi_0}(R_0)$ such that for any klt singularity $x\in (X=\Spec(R),\Delta)$ which degenerates into $x_0\in (X_0,\Delta_0;\xi_0)$ through some $\bR$-test configuration and such that $m\Delta$ is a $\bQ$-Cartier Weil divisor, there exists a divisorial valuation $v_1\in \Val_{X,x}$ that induces a degeneration of $x\in (X,\Delta)$ to the log Fano cone $x_0\in (X_0,\Delta_0;\xi_1)$.
\end{cor}

\begin{proof}
By the above construction and Lemma \ref{lem:lift valuation}, we see that there exists an open neighbourhood $U\subseteq \ft_\bR^+$ of $\xi_0$, depending only on $c_1$ and $\Phi_{\xi_0}(R_0)$, such that for any $\xi_1\in U$ there exists a quasi-monomial valuation $v_1$ of $R$ whose associated graded ring $\gr_{v_1} R$ is isomorphic to $R_0$, and the induced Reeb vector on $X_0=\Spec(R_0)$ is exactly $\xi_1$. If $\xi_1$ is rational, then $v_1$ has rational rank one and hence is divisorial. Since $m\Delta_0$ is Cartier, the same holds for $m\Delta$ at $x$. As $\wt_{\xi_0}(m\Delta)$ belongs to the bounded discrete set $\Phi_{\xi_0}(R_0)\cap [0,c_2]$, there are only finitely possible values of $\wt_{\xi_0}(m\Delta)$. By Lemma \ref{lem:uniform nbh}, this implies that after possibly shrinking $U$ (in terms of $c_1,c_2,m$ and $\Phi_{\xi_0}(R_0)$), we can further assume that the Cartier divisor $m\Delta$ still specializes to $m\Delta_0$ under the degeneration induced by $v_1$. The desired statement then holds for any rational Reeb vector $\xi_1\in U$.
\end{proof}

\subsection{Mld of Koll\'ar components}

We now put the construction from the previous subsection in families to prove Theorem \ref{thm:mld^K bdd}. We need another auxiliary lemma.

\begin{lem} \label{lem:stable degeneration Q-Gor}
Let $x\in (X=\Spec(R),\Delta)$ be a klt singularity such that $K_X$ is $\bQ$-Cartier. Then the same is true for its K-semistable log Fano cone degeneration.
\end{lem}

\begin{proof}
This is a direct consequence of the stable degeneration theory. Let $v\in \Val_{X,x}$ be the minimizer of the normalized volume function. By \cite[Lemma 3.4]{XZ-SDC}, there exists a log smooth model $\pi\colon (Y,E)\to (X,\Delta)$ such that $v$ is a monomial lc place of some special $\bQ$-complement $\Gamma$ with respect to $(Y,E)$ (we refer to Definition 3.1 of {\it loc. cit.} for the relevant definitions). Since $K_X$ is $\bQ$-Cartier, $\Gamma$ is also a special $\bQ$-complement of the klt singularity $x\in X$. Since the underlying singularity of the K-semistable log Fano cone is $X_0=\Spec(\gr_v R)$, we deduce from \cite[Theorem 4.1]{XZ-SDC} that it is also $\bQ$-Gorenstein (and in fact klt).  
\end{proof}

\begin{proof}[Proof of Theorem \ref{thm:mld^K bdd}]

By \cite[Lemma 2.18]{Z-mld^K-2}, it suffices to prove the theorem when $I\subseteq \bQ$ and thus we may assume that $I=\frac{1}{m_0}\bN \cap [0,1]$ for some positive integer $m_0$. By \cite[Proposition 5.8]{Z-mld^K-1}, we also know that it is enough to consider singularities $x\in (X,\Delta)$ for which $K_X$ is $\bQ$-Cartier. These in particular imply that $m_0\Delta$ is a $\bQ$-Cartier Weil divisor.

There exists a bounded set $\cS$ of $\bQ$-Gorenstein log Fano cone singularities with coefficients in $I=\frac{1}{m_0}\bN$ and local volume at least $\varepsilon$. This follows from Theorem \ref{thm-Fanocone bounded} and the fact that in any bounded family $B\subseteq \cX\to B$ of singularities the locus of $b\in B$ for which the fiber $b\in \cX_b$ is $\bQ$-Gorenstein klt is constructible ({\it cf.} the proof of Lemma \ref{lem:two defn of bdd agree}). By Lemma \ref{lem:stable degeneration Q-Gor}, the K-semistable log Fano cone degeneration of any $n$-dimensional $\bQ$-Gorenstein klt singularity $x\in (X,\Delta)$ with ${\rm Coeff}(\Delta)\subseteq I$ and $\hvol(x,X,\Delta)\ge \varepsilon$ belongs to $\cS$.

Without loss of generality, we shall assume that $\cS$ consists of just one flat family $B\subseteq(\cX=\Spec(\cR),\cD)\to B$ over a smooth base where $K_\cX$ is $\bQ$-Cartier (in practice this means we freely stratify the base $B$ and analyze one deformation family at a time). Let $\bT$ be the torus that acts fiberwise on the family. Note that the Reeb cone and the volume function on it is independent of $b\in B$. This is because if $\cR=\oplus \cR_\alpha$ is the weight decomposition, then each $\cR_\alpha$ is free over $\cO_B$ by the flatness of $\cX\to B$, hence their ranks are constant in $b\in B$. After possibly stratifying the base $B$, we may assume (by e.g. \cite[Theorem 2.15(3)]{LX-stability-higher-rank}) that the log discrepancy function $A_{\cX_b,\cD_b}$ on the Reeb cone is independent of $b\in B$ as well. This implies that there exists some $\xi\in N(\bT)_\bR$, independent of $b\in B$, that minimizes the normalized volume function on the Reeb cone. In particular, if $b\in (\cX_b,\cD_b)$ is a K-semistable log Fano cone, then (up to rescaling) $\xi$ is the associated Reeb vector. The value semigroup $\Phi:=\Phi_{\xi} (\cR_b)$ is thus also constant in $b$.

Since $\cS$ is bounded, we may choose some $c>0$ such that for all $b\in B$, the syzygy complexity of $\cR_b$ is at most $c$. By construction, the Weil divisor $m_0\cD$ is $\bQ$-Cartier (since both $K_{\cX}+\cD$ and $K_{\cX}$ are), thus by \cite[Corollary 1.4]{XZ-minimizer-unique} we see that there exists some positive integer $m_1$ depending only on $\varepsilon$ such that $m\cD_b$ is Cartier for all $b$, where $m=m_0 m_1$. Again since $\cS$ is bounded, we may further assume that $\wt_\xi (m\cD_b)\le c$ for all $b\in B$ after possibly enlarging $c$. 

By Corollary \ref{cor:lift kc}, we deduce that there exists a (rational) Reeb vector $\xi_1$, independent of $b\in B$, such that for any klt singularity $x\in (X,\Delta)$ whose K-semistable log Fano cone degeneration $(X_0,\Delta_0;\xi_0)$ belongs to $\cS$, there exists a divisorial valuation $v_1\in \Val_{X,x}$ that induces a special degeneration of $x\in (X,\Delta)$ to the log Fano cone $(X_0,\Delta_0;\xi_1)$. Replacing $\xi_1$ by a fixed multiple, we may also assume that its value group equals $\bZ$. In particular, we have $v_1=\ord_E$ for some divisor $E$ over $x\in (X,\Delta)$. Since its induced degeneration $(X_0,\Delta_0)$ is klt, by \cite[Theorem 4.1]{XZ-SDC} we see that $E$ is necessarily a Koll\'ar component over $x\in (X,\Delta)$. On the other hand, by \cite[Lemma 2.58]{LX-stability-higher-rank} we also have $A_{X,\Delta}(E)=A_{X_0,\Delta_0}(\wt_{\xi_1})$, and the right hand side is bounded from above since $(X_0,\Delta_0)$ belongs to a bounded family and the Reeb vector $\xi_1$ is constant in this family. Hence we conclude that $\mld^K(x,X,\Delta)$ is bounded from above and this finishes the proof.
\end{proof}

\bibliography{ref}

\end{document}